\documentclass[10pt, a4paper]{article}

\usepackage[utf8]{inputenc}
\usepackage[T1]{fontenc}

\usepackage{graphicx}
\usepackage{amsmath}
\usepackage{amsthm}
\usepackage{amsfonts}
\usepackage{xcolor}
\usepackage{amssymb}
\usepackage[hyphens]{url}
\usepackage[hidelinks]{hyperref}
\usepackage[hyphenbreaks]{breakurl}
\usepackage{fullpage}
\usepackage{esvect,mathtools}
\usepackage{subfiles}
\usepackage[normalem]{ulem}
\usepackage{enumitem}
\usepackage[numbers,square,sort]{natbib}
\usepackage[capitalize]{cleveref}

\newtheorem{theorem}{Theorem}
\numberwithin{theorem}{section} 
\newtheorem{lemma}[theorem]{Lemma}

\newtheorem{claim}[theorem]{Claim}
\theoremstyle{remark}

\theoremstyle{definition}
\newtheorem{definition}[theorem]{Definition}

\crefname{lemma}{lemma}{lemmas}
\Crefname{lemma}{Lemma}{Lemmas}
\newcommand{\bE}{\mathbb{E}}

\newcommand{\hide}[1]{}

\title{\vspace{-0.9cm} Nearly tight bounds for induced subdivisions}
\author{
Zach Hunter\thanks{Department of Mathematics, ETH Z\"urich, Switzerland. Email: {\tt \{zach.hunter, aleksa.milojevic, patryk.morawski, benjamin.sudakov\}@math.ethz.ch}. Research supported in part by SNSF grant 200021-228014.}
\and Aleksa Milojevi\'c\footnotemark[1] 
\and 
Patryk Morawski\footnotemark[1]
\and Benny Sudakov\footnotemark[1]}
\date{}

\begin{document}

\maketitle
\begin{abstract}
Subdivisions of complete graphs play a central role in combinatorics, having deep connections to structural, extremal, and topological aspects of graph theory. A celebrated conjecture of Mader, proved independently by Bollob\'as and Thomason and by Koml\'os and Szemer\'edi, states that every graph of average degree of order $h^2$ contains a subdivision of $K_h$. 

In this paper, we consider the induced variant of this problem. A theorem of K\"uhn and Osthus implies that, for every fixed graph $H$ and every $s\ge 1$, graphs of sufficiently large average degree contain either a copy of $K_{s,s}$ or an induced subdivision of $H$. However, even for $H=K_h$, the best previous quantitative bounds were far from optimal.

We prove nearly tight bounds for forcing induced subdivisions of $K_h$. We show that every $K_{s,t}$-free graph of average degree $\Omega_{s,t}(h^{2(s-1)}\log^{7(s-1)} h)$ contains an induced subdivision of $K_h$, and that every $C_{2k}$-free graph with $k \geq 3$ and average degree $\Omega_k(h\log^5 h)$ contains an induced subdivision of $K_h$. These bounds substantially improve the previously known results and are nearly optimal in both settings. They
also hold if $K_h$ is replaced by any other graph on $h$ vertices.
\end{abstract}

\section{Introduction}
A \emph{subdivision} of a graph $H$ is a graph obtained from $H$ by replacing certain edges of $H$ by internally vertex-disjoint paths. Subdivisions play a central role in graph theory, and their study has a very long history, going back to some of its earliest structural results. One of the most classical examples is Kuratowski's theorem, which states that a graph is planar if and only if it contains no subdivision of $K_5$ or $K_{3,3}$. Thus, subdivisions arise naturally as the basic obstructions to planarity. 

Since planar graphs are $4$-colorable by the Four Color Theorem, it is very natural to ask whether avoiding subdivisions of complete graphs might more generally control the chromatic number of a graph. This point of view led Haj\'os to conjecture that every graph of chromatic number at least $r$ contains a subdivision of $K_r$. Although Haj\'os' conjecture was eventually shown to be false, it had enormous influence on the development of the subject and helped establish clique subdivisions as a central object of study. It is also closely related to Hadwiger's conjecture, one of the central problems in graph theory, which asserts that every graph of chromatic number at least $r$ contains a $K_r$-minor.

A different and highly influential direction was initiated by Mader, who studied the relationship between the average degree of a graph and the appearance of subdivisions. In 1967, Mader proved that for every graph $H$ there exists $d=d(H)$ such that every graph of average degree at least $d$ contains a subdivision of $H$~\cite{Mader67}. In the special case $H=K_h$, Mader, and independently Erd\H{o}s and Hajnal, conjectured that average degree of order $h^2$ should already force a subdivision of $K_h$. This was proved by Bollob\'as and Thomason~\cite{BollobasThomason98} and independently by Koml\'os and Szemer\'edi~\cite{KomlosSzemeredi96}. 

Over the years, much work has been devoted to understanding how this picture changes in sparse graph classes. The initial question in this direction, raised by Mader, asks whether every $C_4$-free graph of average degree at least $ch$ must contain a subdivision of $K_h$ for some absolute constant $c>0$. More generally, one can ask how large a clique subdivision must appear in graphs of high average degree which exclude a fixed graph $H$. 

A number of very interesting results in this direction were obtained by K\"uhn and Osthus, who showed, for example, that every graph of sufficiently large girth contains a clique subdivision whose order is linear in its average degree \cite{KuhnOsthus02, KuhnOsthus06Girth}. Moreover, something even stronger is true: a graph $G$ of sufficiently large girth contains a $K_{\delta(G)+1}$-subdivision, where $\delta(G)$ is the minimum degree of a vertex in $G$.  K\"uhn and Osthus also studied clique subdivisions in $C_4$-free graphs~\cite{KuhnOsthus04C4}, and the corresponding extremal problem in $K_{s,t}$-free and more general bipartite graphs~\cite{KuhnOsthus06Bip}. Their results were later tightened by Liu and Montgomery~\cite{LiuMontgomery17}, who proved Mader's conjecture fully and, more generally, showed that for every $s,t\ge 2$ there exists $c=c(s,t)>0$ such that every $K_{s,t}$-free graph of average degree at least $c h^{\frac{2(s-1)}{s}}$ contains a subdivision of $K_h$. Let us also mention that Balogh, Liu and Sharifzadeh \cite{BLS14} showed that $C_6$-free graphs of average degree $ch$ already contain subdivisions of $K_h$.

In recent years, a new trend of studying \textit{induced variants} of classical extremal problems emerged. One source of motivation comes from geometry: many intersection graphs of geometric objects avoid certain induced subgraphs, which can be used as a stepping stone to obtaining novel extremal results about these graphs \cite{HMST25}. For example, the intersection graphs of curves in the plane (known as \textit{string graphs}) avoid induced subdivisions of $K_5$ where every edge is subdivided at least once. 

Induced subgraphs also appear naturally in the context of $\chi$-bounded graph classes. A graph class is said to be $\chi$-bounded if there is a function $f$ such that $\chi(G)\leq f(\omega(G))$ for each graph $G$ in the class, where $\omega (G)$ denotes the clique number of $G$. One of the most important questions related to $\chi$-boundedness is the Gyarfas-Sumner conjecture, which asks whether the class of graphs avoiding induced copies of a fixed tree $T$ is $\chi$-bounded \cite{Gyarfas,Sumner}. In his PhD thesis, Scott showed that the class of graphs avoiding induced subdivisions of a fixed tree $T$ is $\chi$-bounded, and asked whether the same could be true if the tree $T$ is replaced by a general graph $H$. While this was disproved by the construction of Pawlik, Kozik, Krawczyk, Laso\'n, Micek, Trotter and Walczak~\cite{PawlikKozikKrawczykLasonMicekTrotterWalczak14}, understanding the quantitative aspects of this question remains interesting \cite{NguyenScottSeymour}.

A fundamental theorem of K\"uhn and Osthus~\cite{KuhnOsthus04Induced} states that for every graph $H$ and every integer $s\ge 1$, there exists $d=d(H,s)$ such that every graph of average degree at least $d$ contains either a copy of $K_{s,s}$ or an induced subdivision of $H$. Thus, once one excludes large complete bipartite subgraphs, graphs with no induced subdivision of a fixed graph form a degree-bounded class. Note that it is indeed necessary to exclude $K_{s, s}$ if one hopes to find any nontrivial induced subgraphs, since otherwise the host graph could simply be a complete bipartite graph, which does not have any interesting induced subgraphs. For a long time, however, essentially no reasonable quantitative bounds were known for this theorem of K\"uhn and Osthus. Motivated in part by this gap, Bonamy, Bousquet, Pilipczuk, Rzążewski, Thomass\'e, and Walczak \cite{BBPRTW22} conjectured that when $s$ grows, polynomial average degree in $s$ should already force an induced $K_{h}$-subdivision in $K_{s, s}$-free graphs.

Recently, there has been substantial progress on this problem. Du, Gir\~ao, Hunter, McCarty and Scott~\cite{DuGiraoHunterMcCartyScott25} proved that there exists a constant $C>0$ such that for all $s,h\in\mathbb{N}$, every $K_{s,s}$-free graph of average degree $h^{Cs^3}$ contains induced subdivisions of $K_h$, giving the first polynomial bound for this problem. Shortly afterwards, Gir\~ao and Hunter~\cite{GiraoHunter25} and Bourneuf, Buci\'c, Cook and Davies~\cite{BBCD24} studied the regime where $s$ is large compared to $H$, and proved that for every graph $H$ and every $s\ge 2$, every $K_{s, s}$-free graph of average degree $\Omega_H(s^{A|V(H)|^2})$ contains an induced subdivision of $H$ (where $A>0$ is an absolute constant).

The main goal of this paper is to understand this question in the regime where $h$ is growing much more precisely. Namely, we ask the following.

\begin{quote}
\emph{What average degree in an $F$-free graph guarantees an induced subdivision of $K_h$?}
\end{quote}

We essentially resolve this problem in two most natural sparse settings: $K_{s,t}$-free graphs and $C_{2k}$-free graphs, thinking of $s, t$ and $k$ as constants. We show that in both cases one can force induced subdivisions of complete graphs with bounds that are nearly best possible, up to polylogarithmic factors.

\begin{theorem}\label[theorem]{theorem:kst}
Let $h \geq t \geq s\geq 2$, let $H$ be a graph on $h$ vertices and let $G$ be a $K_{s,t}$-free graph with average degree $\Omega_{s, t}(h^{2(s-1)}\log^{7(s-1)} h)$. Then $G$ contains an induced subdivision of $H$.
\end{theorem}

\begin{theorem}\label[theorem]{theorem:C2k}
Let $H$ be a graph on $h$ vertices and let $G$ be a $C_{2k}$-free graph with average degree $\Omega_k(h\log^5 h)$, where $k\geq 3$ is a positive integer. Then $G$ contains an induced subdivision of $H$.
\end{theorem}

\noindent
Note that Theorem~\ref{theorem:C2k} can be strengthened if, instead of forbidding $C_{2k}$, one assumes that $G$ has large girth. In this case, Gir\~ao and Hunter \cite{GiraoHunter26} recently proved that every graph of minimum degree $\delta\geq 10^8$ and girth at least $10^8$ contains an induced subdivision of $K_{\delta+1}$. 

Let us now explain why the bounds from Theorems~\ref{theorem:kst} and~\ref{theorem:C2k} are nearly optimal. For $C_{2k}$-free graphs, the obstruction is essentially trivial: every subdivision of $K_h$ has $h$ branch vertices, and each branch vertex must have degree at least $h-1$ in the host graph. Consequently, average degree of order at least $h$ is a necessary condition for forcing an induced subdivision of $K_h$. Thus, Theorem~\ref{theorem:C2k} is optimal up to the factor of $\log^5 h$. 

Interestingly, our results do not fully explain what happens in $C_4$-free graphs. In this case, the upper bound (coming from Theorem~\ref{theorem:kst}) is $h^2(\log h)^{O(1)}$. However, the best lower-bound construction we have comes from a polarity graph over $\mathbb{F}_q^2$. Since this graph contains no independent sets much larger than $q^{3/2}$, it contains no induced subdivisions of $K_h$ for $h\gg q^{3/4}$. This shows that we must have $d\gg h^{4/3}$ in order to force an induced $K_h$-subdivision in a $C_4$-free graph (see Section~\ref{sec:concluding_remarks} for further comments).

For the $K_{s,t}$-free case, the near-optimality comes from a random construction. If $G\sim G(n, p)$ where $n= h^{2st/(s+t)-o(1)}$ and $p=n^{-1/s-1/t}$, a simple union bound implies that $G$ is $K_{s, t}$-free with good probability, since $\binom{n}{s}\binom{n}{t}p^{st}\leq \frac{n^s}{s!}\frac{n^t}{t!}n^{-s-t}\leq \frac{1}{s!t!}$. The crucial observation is that if $G$ contains an induced subdivision of $K_h$, then choosing one internal vertex on each path corresponding to a subdivided edge produces an independent set. Hence, since every set of $h$ vertices spans at most $\frac{1}{2}\binom{h}{2}$ edges, at least half of the edges in any $K_h$ must be subdivided and so $\alpha(G)\ge \binom{h}{2}/2$. On the other hand, $\alpha(G(n,p))=\tilde \Theta(\frac{1}{p})\ll h^2$, showing that this is impossible. The graph produced in this way has average degree of order $pn=n^{1-1/s-1/t}=h^{2(s-1)-O(s^2/t)}$, showing that the bound in Theorem~\ref{theorem:kst} is very close to best possible when $t$ is large. 

\medskip
\noindent
\textbf{Notation.} Now we establish basic notation and terminology which will be used throughout the paper.  The \textit{ball of radius $i$} around a vertex $u$ (or a set of vertices $U\subseteq V(G)$) is defined as the set of all vertices at distance at most $i$ from $u$ (or $U$), and it is denoted by $B_G^{(i)}(u)$ (or $B_G^{(i)}(U)$).
The first neighbourhood of a vertex $v$, denoted by $N_G(v)$, is the set of all vertices adjacent to it. Similarly, the $i$-th neighbourhood $N_G^{(i)}(v)$ is defined as $N_G^{(i)}(v)=B^{(i)}(v) \setminus B^{(i-1)}(v)$.
For two vertices $u, v\in V(G)$, we write $d_G(u, v)$ for the size of their common neighbourhood.
We might omit the subscript $G$ and write $N(v)$ for $N_G(v)$ etc., whenever the graph $G$ is clear from the context.
Further, we say that $H$ is a proper subdivision of $G$ if $H$ can be obtained from $G$ by replacing each edge by a path of length at least $2$. 
We let $P_k$ denote the path with $k$ edges, i.e., on $k+1$ vertices.
Finally, if $P, P'$ are paths, then $P^{-1}$ denotes the reverse of the path $P$, while $PP'$ denotes the concatenation of these two paths (which makes sense only if the endpoint of $P$ coincides with the starting point of $P'$).

\medskip
\noindent
\textbf{Organization of the paper.} We begin by giving an outline of our proofs in Section~\ref{sec:outline}. In Section~\ref{sec:irregular} we will prove a dichotomy statement, which shows that it suffices to deal with nearly-regular and highly irregular graphs separately. In the same section, we also address the case of highly irregular graphs. In Section~\ref{sec:nearly_regular} we then use sublinear expanders to find subdivisions in nearly-regular graphs. We then assemble the proofs of the main theorems in Section~\ref{sec:main_proofs}. Finally, in Section~\ref{sec:concluding_remarks} we discuss some further directions left open by this work.

\section{Proof outline}\label{sec:outline}
The proofs of Theorems~\ref{theorem:kst} and \ref{theorem:C2k} follow the same general strategy.
We therefore sketch the proof of Theorem~\ref{theorem:kst} and explain the necessary modifications for Theorem~\ref{theorem:C2k} at the end.
We fix $h$ and $t \geq s$ and let $G$ be a $K_{s, t}$-free graph with average degree $d = \Omega(h^{2(s-1)}\log^{7(s-1)} h)$.
Since the average degree is the relevant parameter throughout, we may assume that $G$ is $d$-degenerate and has minimum degree at least $d/2$.
Otherwise, we can pass to a denser subgraph of $G$ and show a stronger statement there.
We also note that it is sufficient to argue that we can find an induced proper subdivision of $K_h$, i.e., an induced subdivision where each edge of $K_h$ is subdivided at least once. 
Indeed, such a proper subdivision contains a subdivision of any graph $H$ on $h$ vertices as an induced subgraph.

Our argument proceeds in several largely independent steps.
First, we apply a structural dichotomy: either $G$ looks like a very unbalanced bipartite graph or $G$ contains a nearly-regular induced subgraph with almost the same average degree.
We treat these two cases separately.
In the former case, we reduce the problem to finding a not-necessarily induced subdivision in an auxiliary graph with average degree $\Omega(h^2)$.
In the latter case we first pass to an induced $K_{2, \log d}$-free subgraph with average degree $\Omega(h^2\log^3h)$, thereby reducing to the case of $s = 2$.
Then, we find an induced subdivision directly, using sublinear expansion, following an approach similar to the one of Koml\'os and Szemer\'edi \cite{KomlosSzemeredi96}, with several crucial modifications.\footnote{Upon reading the article \cite{KomlosSzemeredi96} in detail, we found what we believe to be a gap in their approach. Namely, we believe that in Section 3.3 of \cite{KomlosSzemeredi96}, the equation (12) cannot hold for all graphs $G$, and the error arises when the preceding paragraph claims $\lceil p/L\rceil \leq 2p/L$, which may not be the case if $L\gg p$ (which can happen for some graphs). This ultimately leads to the problem of exhausting the neighbourhoods of the branch vertices too quickly. We believe one can go around this issue by using our notion of drifting-away paths, introduced in Section~\ref{sec:nearly_regular} of this article (and probably in other ways too).} We now describe the main steps in more detail.

\medskip
\noindent\textbf{The Dichotomy:} We use the following dichotomy result. 
For every parameter $K \geq 1$ either there is a partition $V(G) = A \cup B$ with $|A| \geq K|B|$ such that a constant proportion of the edges of $G$ go between $A$ and $B$, or $G$ contains a nearly-regular induced subgraph with average degree $\Omega(d/\log^2K)$.
Here, nearly-regular means that the maximum degree is bounded by some absolute constant times its average degree.
We analyze the two cases separately.

\medskip
\noindent\textbf{The highly irregular case:} In the first case of the dichotomy, using that both $G[A]$ and $G[B]$ are $d$-degenerate, we will be able to find independent sets $A' \subseteq A$ and $B' \subseteq B$ such that $|N(a) \cap B'| = 2$ for each $a \in A'$ and such that the auxiliary graph $H$ on the vertex set $B'$ and an edge $N(a)\cap B'$ for each $a \in A'$ has average degree at least $\Omega(h^2)$.
By the classical result of Bollob\'as--Thomason and Koml\'os--Szemer\'edi, $H$ therefore contains a subdivision of $K_h$ --- and it is easy to see that this subdivision corresponds to an induced subdivision of $K_h$ in $G$.

\medskip
\noindent\textbf{Subsampling to get a $K_{2, \log d}$-free induced subgraph:} We now turn to the nearly-regular case.
Here, we first observe that while finding induced $C_4$-free subgraphs in $K_{s,t}$-free graphs is potentially difficult, finding $K_{2, \log d}$-free ones is relatively easy.
As the dependence on $t$ we get in the $K_{2, t}$-free case is only linear, this will be, up to a loss of one $\log h$, good enough for us.
To find such an induced subgraph we subsample the vertices of $G$ at $p = \Theta(1/(d^{1 - 1/(s-1)}\log d))$ and consider two types of $K_{2, \log d}$'s that can appear in this random induced subgraph of $G$ --- the ones with the two vertices on the smaller side have co-degree $\Omega(d^{1 - 1/(s-1)})$ in $G$ and the ones where this co-degree is $O(d^{1 - 1/(s-1)})$.
For the former case, by K\H{o}v\'ari-S\'os-Tur\'an the number of pairs in $G$ with such high co-degree is $O(nd^{1 - 1/(s-1)})$.
For the latter case, it is very unlikely that $\log d$ common neighbours survive the subsampling.
Therefore, after some cleaning, we will get a $K_{2, \log d}$-free induced subgraph.

\medskip
\noindent\textbf{Finding induced subdivisions using sublinear expansion:}
This step is the main difference of this paper to the previous works on finding induced subdivisions.
While previous authors reduced the problem to finding a not-necessarily induced subdivision of $K_h$ in an auxiliary graph in a similar way to our highly irregular case, here we will find an induced subdivision directly.
We first describe the argument in the case when $n$ is polynomial in $h$; the general case follows by a refinement of the same idea.

First, by losing a constant factor in the average degree, we can assume that $G$ is a sublinear expander.
In particular, for any vertex sets $U_1, U_2$ and any forbidden set $F$ which is sufficiently small compared to $U_1$ and $U_2$, there exists a path of length $O(\log^3n) = O(\log^3 h)$ between $U_1$ and $U_2$ avoiding $F$.

We select $v_1, \dots, v_h$ in $G$ such that the balls of radius $2$ around them are disjoint and form induced trees with $\Omega(d^2)$ leaves.
It is not the exact picture but is close enough to the truth for this sketch.
These vertices will serve as branch vertices of the desired subdivision.

We then connect the pairs $v_i, v_j$ one by one.
At each step, we maintain a forbidden set $F$ consisting of all vertices at distance at most one from the previously constructed paths.
Since each path has length $O(\log^3h)$ and $\Delta(G) = O(d)$, we have that $F$ has size at most $O(d \cdot h^2 \log^3 h)$.
Let $U_i$ and $U_j$ denote the second neighbourhoods of $v_i$ and $v_j$ in the graph $G - F$.
These sets remain large, of size $\Omega(d^2)$, and thus are still sufficiently large compared to $F$.
By sublinear expansion, there exists a short path between $U_i$ and $U_j$ avoiding $F$. By construction, after having connected all the pairs, the paths we find will yield an induced subdivision of $K_h$.

When $n$ is large, we change the argument slightly by ensuring that the connecting paths do not spend too much time near the branch vertices $v_1, \dots, v_h$; we refer to \cref{section:sublinear_expanders} for details.

\medskip
\noindent\textbf{Forbidding longer even cycles:}
The proof of Theorem~\ref{theorem:C2k} follows the same general approach. The main difference arises in the final step. 
In $C_{2k}$-free graphs, with $k \geq 3$, a ball of radius $3$ around a vertex is roughly an induced tree.
We therefore work with third instead of second neighbourhoods.
These have size $\Omega(d^3)$, while the forbidden set remains of size $O(d \cdot h^2\log^3 h)$, which is negligible compared to $\Omega(d^3)$.
So, we can carry out the same expansion-based construction and obtain an induced subdivision of $K_h$.

\section{Highly irregular graphs}\label{sec:irregular}
\subsection{Dichotomy}

The goal of this section is to state and prove our dichotomy lemma that will allow us to assume that our graph is either nearly-regular or looks like a very unbalanced bipartite graph. This result was communicated to the first author by Ant\'onio Gir\~ao and refines variants appearing already in \cite{GiraoHunter25, BBCD24}. 
Here, nearly-regular means that $\Delta(G) \leq C \cdot d(G)$ for some constant $C$.

\begin{lemma}[The dichotomy]\label{lemma:dichotomy}
    There exists an absolute constant $C > 0$ such that the following holds. For any $d \geq 1$, $K > 4$ and any $d$-degenerate graph $G$ on $n$ vertices with $d(G) \geq d$, we have one of the two alternatives
    \begin{enumerate}
        \item there is a partition $V(G) = A \cup B$ where $|A| \geq K|B|$, $e(G[A]) \leq nd/100$ and $e(G[A, B]) \geq nd/8$, \emph{or}
        \item there is an induced subgraph $H \subseteq G$ with $d(H) \geq d/ \left(2000\log^2 K\right)$ and $\Delta(H) \leq Cd(H)$.
    \end{enumerate}
\end{lemma}

The main ingredient in the proof will be the following lemma, which says that every nearly-regular graph $H$ contains an induced subgraph which is nearly-regular with better parameters and whose average degree is not much smaller than the average degree of $H$. We will repeatedly use it starting with a graph with some (bad) bound on the maximum degree until we find a nearly-regular subgraph $G'$ of it with $\Delta(G') \leq Cd(G')$ for some constant $C$.

\begin{lemma}\label{lemma:getting_more_regular}
Let $L \geq e^{20}$ and let $H$ be a graph with $d(H) \geq d$ and $\Delta(H) \leq Ld$. Then, there exists an induced $H' \subseteq H$ with $d(H') \geq d / \left(200\log L\right)$ and $\Delta(H') \leq 5600\log L \cdot d(H')$.
\end{lemma}

Let us show how the dichotomy lemma follows from Lemma~\ref{lemma:getting_more_regular}. We begin by dividing the vertices into two sets, $A_{\geq 4K}$ and $A_{<4K}$, based on whether their degree is larger or smaller than $4Kd(G)$.
Then, either there is a lot of edges between $A_{\geq 4K}$ and $A_{<4K}$, in which case we get our partition for the first case, or there is a lot of edges inside $A_{<4K}$.
In the latter case, we gain a reasonable bound on the maximum degree of any vertex in $A_{<4K}$ -- so we can iteratively apply Lemma~\ref{lemma:getting_more_regular} to find an induced nearly-regular subgraph.

\begin{proof}[Proof of Lemma~\ref{lemma:dichotomy} assuming Lemma~\ref{lemma:getting_more_regular}]
    Let $A_{\geq 4K} = \{v \in V(G): d(v) \geq 4K d\}$ and $A_{<4K} = V(G) \setminus A_{\geq 4K}$.
    Note that since $G$ is $d$-degenerate we have $e(G) \leq nd$ and therefore $|A_{\geq 4K}| \leq n/(2K)$.
    Moreover, $e(G[A_{\geq 4K}]) \leq |A_{\geq 4K}|d \leq nd/8$.
    If $e(G[A_{<4K}]) < nd/100$, then by taking $A = A_{<4K}$ and $B = A_{\geq 4K}$ we immediately get the first case.

    If $e(G[A_{<4K}]) \geq nd/100$ we want to show that the second point holds.
    Let $G_0 = G[A_{<4K}]$ and note that $d(G_0) \geq d/50$.
    By definition we have $d_{G_0}(v) \leq 200Kd(G_0)$ for each $v \in V(G_0)$.

    We now want to iteratively apply Lemma~\ref{lemma:getting_more_regular} to find our induced $H \subseteq G_0$ --- for that we first set up the relevant parameters. 
    Let $C > e^{20}$ be an (absolute) constant such that $200\log C \cdot \log^2(5600\log C) \leq \log^2 C$ and note that this inequality then holds for all $C' \geq C$.
    This will be the threshold for $L$ at which we stop the process.
    We let $L_0 = 200K$ and for each $i \geq 1$ let $L_i = 5600 \log L_{i-1}$.
    Similarly, let $d_0 = d/50$ and for each $i \geq 1$ write $d_i = d_{i-1} / (200 \log L_{i-1})$.
    Let $\ell$ be the smallest integer such that $L_{\ell} \leq C$, which exists since $5600 \log L' \leq L'/2$ for each $L' \geq C$.
    By repeatedly using Lemma~\ref{lemma:getting_more_regular}, starting with $G_0$, we can for each $i \in [\ell]$ find an induced subgraph $G[W] \subseteq G$ with $d(G[W]) \geq d_i$ and $\Delta(G[W]) \leq L_id_i$.

    We now argue by induction that for each $i =\ell-1, \dots, 0$ we have $d_{\ell} \geq d_i / \log^2L_i$.
    First, note that $d_{\ell} \geq d_{\ell -1} / 200\log L_{\ell - 1} \geq d_{\ell - 1} / \log^2 L_{\ell -1}$, since $L_{\ell - 1} \geq C$.
    Suppose then that the bound holds for some $i \in [\ell]$.
    By induction we get 
    \[
        d_\ell \geq \frac{d_{i}}{\log^2L_i} = \frac{d_{i-1}}{200\log L_{i-1} \cdot \log^2(5600 \log L_{i-1})} \geq \frac{d_{i-1}}{\log^2 L_{i-1}}, 
    \]
    where we again used that $L_{i-1} \geq C$.
    
    In particular, we therefore get that $d(G_{\ell}) \geq d_0 / \log^2 L_0 \geq d/2000\log^2 K$.
    By construction, $G_{\ell}$ is an induced $C$-nearly-regular subgraph of $G$.
    Hence, the second case of the lemma holds.
\end{proof}

It therefore remains to prove Lemma~\ref{lemma:getting_more_regular} --- to do that we need one additional ingredient.
Namely, in the proof of Lemma~\ref{lemma:getting_more_regular} we will encounter a bipartite graph $\Gamma$ with parts $A$ and $B$ such that all vertices in $A$ have roughly the same degree and all vertices in $B$ have roughly the same degree --- the degrees in the two parts might however be different.
The following lemma says that we can subsample one of the parts and make all vertices have roughly the same degree.
\begin{lemma}\label{lemma:subsampling_biregular}
    Let $L \geq 1$ and let $\Gamma = (A, B; E)$ be a bipartite graph with $d(a) \leq L \cdot |E| / |A|$ for each $a \in A$ and $d(b) \leq L \cdot |E|/ |B|$ for each $b\in B$. 
    Then, $\Gamma$ has an induced subgraph $\Gamma'$ with $d(\Gamma') \geq d(\Gamma)/4$ and $\Delta(\Gamma') \leq 12Ld(\Gamma')$.
\end{lemma}
\begin{proof}
    We can without loss of generality suppose that $|A| \geq |B|$ and that $|E| > 0$.
    Write $d = |E|/|A|\ge d(\Gamma)/2$ and set $p = |B| / |A|$. 
    Let $A' \subseteq A$ be a random subset of $A$ where each vertex is included independently with probability $p$.
    Moreover, write $B' = \{ b\in B: |N(b)\cap A'| \leq 1 + 2p(d(b) - 1)\}$.
    We take $\Gamma' = \Gamma[A', B']$ and show that it satisfies the desired properties with positive probability.

    To that end, we first note that by construction we have $\Delta(\Gamma') \leq 1 + 2L\frac{|E|}{|A|} \leq 3L\frac{|E|}{|A|} \leq 3L d(\Gamma)$.
    Indeed, for each $a \in A'$ we have $d_{\Gamma'}(a) \leq d_{\Gamma}(a) \leq L\frac{|E|}{|A|}$ and for each $b \in B'$ we have $d_{\Gamma'}(b) \leq 1+ 2pL\frac{|E|}{|B|}$ by definition of $B'$.
    What remains to show is therefore that $d(\Gamma') \geq d(\Gamma)/4$ with positive probability.

    For that, let us show that the probability an edge $e = ab \in E$ is included in $\Gamma'$ is at least $p/2$. To see this, recall that $\Pr[a\in A']=p$. Given that $a\in A'$, the vertex $b$ is included in $B'$ if and only if at most $2p(d(b)-1)$ vertices from the set $N_{\Gamma}(b) \setminus \{a\}$ are included in $A'$ (and these decisions are independent from whether $a\in A'$). If we denote by $X=\big|(N_{\Gamma}(b) \setminus \{a\})\cap A'\big|$ the random variable which counts such vertices, then $\bE[X]=p(d(b)-1)$ and by Markov's inequality we get that
    \[\Pr\big[b \in B'| a \in A'\big] = \Pr\big[X \leq 2\mathbb{E}[X]\big] \geq 1/2.\]
    Hence, $\Pr[e\in E(\Gamma')]=\Pr[a\in A']\Pr\big[b \in B'| a \in A'\big]\geq p/2$
    Consequently, $\mathbb{E}[e(\Gamma')] = \sum_{e \in E} \Pr[e\in E(\Gamma')] \geq p|E|/2$.
    Therefore
    \[
        \mathbb{E}\Big[4e(\Gamma') - \frac{|E|}{|A|} (|A'| + |B'|)\Big] \geq 4\frac{p|E|}{2} - 2\frac{|E|}{|A|}|B| = 0,
    \]
    where we used that $\mathbb{E}\big[|A'| + |B'|\big] \leq p|A| + |B| = 2|B|$.
    In particular, there must exist a choice of $A', B'$ such that for the corresponding $\Gamma'$ we get
    \[
        d(\Gamma') = \frac{2e(\Gamma')}{|A'| + |B'|}\geq \frac{|E|}{2|A|} \ge d(\Gamma)/4,
    \]
    as desired. From this we deduce that $\Delta(\Gamma')\leq 3L d(\Gamma)\leq 12Ld(\Gamma')$.
\end{proof}

With this in hand we are ready to conclude the dichotomy section with the proof of Lemma~\ref{lemma:getting_more_regular}.
\begin{proof}[Proof of Lemma~\ref{lemma:getting_more_regular}]
    We may assume that $H$ is $d$-degenerate and that $\delta(H) \geq d/2$.
    Otherwise we can pass to an induced subgraph $H'$ of $H$ which maximizes the average degree, which is $d(H')$-degenerate and has $\delta(H') \geq d(H')/2$, and show the statement for $H'$.
    
    Let $V(H) =  A \cup B$ be a partition maximizing $e(H[A, B])$ and assume that $|A| \geq |B|$.
    Let $A' =\{ a \in A : d(a) \leq 10d \}$ and note that $|A'| \geq |A|/2$ --- otherwise we would get that $d(H) \geq \frac{1}{2|A|}10d |A\setminus A'|  > 2d$, a contradiction to $H$ being $d$-degenerate.
    We also note that by our choice of the partition we have $e(H[A', B]) \geq \sum_{a\in A'}d(a)/2\ge |A'| d/4$, where we used that $d(a)\geq d/2$ for all $a\in V(H)$.

    We now want to find $B' \subseteq B$ such that all vertices in $B'$ have roughly the same degree into $A'$ and the average degree in $H[A', B']$ does not decrease by much compared to $H[A', B]$.
    To that end, write $d_B = |A'|d / (4|B|) \geq d/8$ and for each $i \in [\lceil \log L\rceil + 5]$ let $B_i = \{ b \in B: 2^{i-2}\cdot d_B \leq d(b, A') < 2^{i-1} \cdot d_B \}$.
    Moreover, let $B^{-} =\{ b\in B: d(b, A') < d_B /2 \}$ and notice that since $d(b) \leq Ld \leq 8Ld_B$ for every $b \in B$ by assumption, $B^-, B_1, \dots, B_{\lceil \log L\rceil + 5}$ forms a partition of $B$. 
    Since $e(A', B^{-}) \leq |B| \cdot d_B/2 \leq e(A', B)/2$, by Pigeonhole principle there exists an $i \in [\lceil \log L\rceil + 5]$ such that $e(A', B_i) \geq e(A', B)/\left(2(\lceil \log L\rceil +5)\right) \geq |A'|d/\left(15\log L\right)$.
    Fix this index $i$, let $B' = \{ b \in B_i : d(b, B_i) \leq 10 d\}$.
    Since $H$ is $d$-degenerate we must have $|B_i \setminus B'| \leq |B_i| / 5$ and so $e(A', B_i\backslash B')\leq 2^{i-1}d_B\cdot  |B_i\backslash B'|\leq 2^{i-2}d_B|B_i|\cdot \frac{2}{5}\leq \frac{2}{5}e(A', B_i)$. Hence,
    \[
        e(A', B') \geq e(A', B_i)-e(A', B_i\backslash B')  \geq e(A', B_i) - \frac{2}{5}e(A', B_i) \geq \frac{|A'|d}{30\log L}.
    \]

    We have thus found sets $A', B'$ such that all vertices in $A'$ have roughly the same degree and all vertices in $B'$ have roughly the same degree -- we would now like to make sure that this degree is roughly the same for both sides.
    To that end, let $\Gamma = H[A', B']$ be the induced bipartite graph between $A'$ and $B'$.
    We have $d(\Gamma) = 2e(A', B')/(|A'| + |B'|) \geq d/50\log L$,
    where we used that $|B'| \leq |B| \leq 2|A'|$.
    Moreover, note that we have $d_{\Gamma}(a) \leq 10d \leq 300\log L \cdot e(\Gamma) / |A'| $ for each $a \in A'$ and by the choice of $B' \subseteq B_i$ we must have $d_{\Gamma}(b) \leq 2e(\Gamma)/|B'|$ for each $b \in B'$.
    Therefore, we can apply Lemma~\ref{lemma:subsampling_biregular} with the parameter $300 \log L$ to find $A'' \subseteq A'$ and $B'' \subseteq B'$ such that for $\Gamma' = H[A'', B'']$ we have $d(\Gamma') \geq d(\Gamma)/4 \geq d/200\log L$ and $d_{\Gamma'} (v) \leq 3600 \log L \cdot d(\Gamma')$. 

    Finally, we want to argue that taking $H' = H[A'' \cup B'']$ satisfies the conditions of the lemma.
    Indeed, we have $d(H') \geq d(\Gamma') \geq d/200\log L$ and by the earlier choice of $A', B'$ we have $d_{H'}(v) \leq d_{\Gamma'}(v) + 10d  \leq 5600\log L \cdot  d(H')$ for each $v \in A'' \cup B''$.  
\end{proof}

\subsection{Finding induced subdivisions in highly irregular graphs}
In this section, we want to argue that a highly irregular graph $G$ --- i.e. a graph for which the first case of Lemma~\ref{lemma:dichotomy} holds --- contains induced subdivisions. 
Our general approach will be to define a suitable auxiliary graph $H$ such that finding an induced subdivision in $G$ reduces to finding a not-necessarily induced subdivision in $H$.
For the latter part, we will then invoke the classical result of Bollob\'as and Thomason, and independently of Koml\'os and Szemer\'edi.

\begin{theorem}[\cite{BollobasThomason98, KomlosSzemeredi96}]\label[theorem]{theorem:noninduced_subdivisions}
    Let $h \in \mathbb{N}$ and let $H$ be a graph with average degree at least $10^4h^2$.
    Then $H$ contains a subdivision of $K_h$.
\end{theorem}

To construct the graph $H$ we do as follows.
Starting from the partition $V(G) = A \cup B$ with $|A| \gg |B|$ we find independent sets $A' \subseteq A$ and $B' \subseteq B$ such that: 1) $|A'| \geq 10^4h^2|B'|$, 2) each vertex in $A'$ has exactly two neighbours in $B'$, and 3) the pairs $N(a) \cap B'$ are distinct for each $a \in A'$. 
We then define the graph $H$ to have the vertex set $B'$ and add the edge $N(a) \cap B'$ for each $a \in A'$.
It is easy to see that a subdivision of $K_h$ in this $H$ gives us an induced subdivision of $K_h$ in $G$.

We note that since we can assume that $G$ is $d$-degenerate, finding large independent sets inside $A'$ and $B'$ is not difficult --- the task is thus to make the sets have the latter two properties.
We find the sets in two different ways depending on whether for each vertex in $A$ we can find a pair in its neighbourhood in $B$ which is not shared with a lot of other vertices in $A$ or not.
In the former case, we show that subsampling both sides at probability $1/d$ gives us the desired sets.
In the latter case, we will, roughly speaking, restrict ourselves to an independent set $B' \subseteq N(a) \cap B$ inside a neighbourhood of a bad vertex $a$ and construct $A'$ greedily.

Without further ado, let us prove that we can indeed find induced subdivisions in highly irregular graphs without a $K_{s, t}$ or without a $C_{2k}$.
Since the proofs for both cases follow the same line, we combine them both into one lemma.
\begin{lemma}\label{lemma:lopsided_case}
    Let $h \in \mathbb{N}$ and let $G$ be a $d$-degenerate graph with a partition $V(G) = A \cup B$ such that $|A| \geq K|B|$ and for each $a \in A$ we have $d(a) \leq 10^3d$ and $d(a, B) \geq d/10^3$.
    Then
    \begin{enumerate}
        \item Let $t \geq s \geq 2$ and suppose that $d \geq (10s)^{6s}th^{2(s-1)}$ and $K \geq t^2s^2d^{200s}$.
        Suppose moreover that $G$ is $K_{s, t}$-free.
        Then, $G$ contains an induced proper subdivision of $K_h$.
        \item Let $k \geq 2$ and suppose that $d \geq 10^{10}k^2h$ and $K \geq d^{200}$.
        Suppose moreover that $G$ is $C_{2k}$-free.
        Then, $G$ contains an induced proper subdivision of $K_h$.
    \end{enumerate}
    \begin{proof}
        We split the proof into two cases.

        \medskip
        \noindent
        \textbf{Case 1:} For some $a \in A$, every non-edge $\{b_1, b_2\} \in \binom{N(a) \cap B}{2}$ has $|N(b_1) \cap N(b_2) \cap A| \geq \sqrt{K}$.

        We first argue that this in fact cannot happen in the second part of the lemma, if $G$ is $C_{2k}$-free.
        Indeed, suppose such an $a \in A$ exists.
        We note that $N(a) \cap B$ is $P_{2k-1}$-free and in particular there exists an independent set $I \subseteq N(a) \cap B$ of size at least $|I| \geq d/(10^3\cdot2k) \geq k$. Let $b_1, \dots, b_k\in I$ be arbitrary vertices, and note that for any $1\leq i\leq k$ the vertices $b_i, b_{i+1}$ (where the indices are considered modulo $k$) have at least $\sqrt{K}>k$ common neighbours in $A$. The goal is then to choose a common neighbour $a_i\in N(b_i)\cap N(b_{i+1})\cap A$ for each $i$, such that $a_1, \dots, a_k$ are distinct. This can be done greedily: having picked $a_1, \dots, a_i$, we have forbidden at most $i<k$ vertices from being chosen as $a_{i+1}$, and so we can continue the process. Since $b_1a_1\dots b_ka_k$ is a cycle of length $2k$, we have a contradiction.

        If $G$ is $K_{s, t}$-free, let us show how to find an induced subdivision of $K_h$ in this case.
        To that end, we write $N = N(a) \cap B$ and define the set $F\subseteq A$ of vertices we will avoid using as $F = \{ a' \in A : |N(a') \cap N| \geq s \}$.
        Note that since $G$ is $K_{s, t}$-free we must have $|F| < t\binom{|N|}{s} \leq td^{10s}$.
        In particular, by our assumption, for each non-edge $\{b_1, b_2\} \in \binom{N}{2}$ we must have $|N(b_1) \cap N(b_2) \cap (A\setminus F)| \geq \sqrt{K} - |F| \geq s^2d^{10}$.

        Let now $I \subseteq N$ be an independent set of size $10^6s^2h^2$, which exists since $G[N]$ is $K_{s-1, t}$-free and so
        \[
            \alpha(G[N]) \geq \frac{|N|}{d(G[N])+1}\geq \frac{1}{2}\Big(\frac{|N|}{t}\Big)^{1/(s-1)} \geq \frac{1}{2}\Big(\frac{d}{10^3t}\Big)^{1/(s-1)} \geq 10^6s^2h^2.
        \]
        The second inequality above follows from the K\H{o}v\'ari-S\'os-Tur\'an theorem, which states that a $K_{s-1,t}$-free graph on $|N|$ vertices has average degree at most $2t^{1/(s-1)}|N|^{1-1/(s-1)}$.
        We note that in particular, this implies that we have $s \geq 3$, as otherwise a pair of vertices in $I$ together with their $\sqrt{K}\geq t$ common neighbours in $A$ would form a $K_{2,t}$.

        For each pair $\{b_1, b_2\}$ we want to pick a  $a^{(b_1, b_2)} \in N(b_1) \cap N(b_2) \cap (A\setminus F)$ such that all $a^{(b_1, b_2)}$'s are distinct and form an independent set.
        We can do that greedily --- after having picked some of the $a^{(b_1, b_2)}$'s the number of vertices in $A$ at distance at most one to one of them is at most
        \[
            (10^6 s^2 h^2)^2 \cdot (1 + 10^3d) \leq s^2d^{10}/2 \leq |N(b_1') \cap N(b_2') \cap (A \setminus F)|/2 
        \]
        for any $\{b_1', b_2'\} \in \binom{I}{2}$.

        We have thus almost found an induced subdivision of $K_{|I|}$ in $G$ --- except that the vertices $a^{(b_1, b_2)}$ can have more than $2$ neighbours in $I$ each.
        To circumvent that, let $B' \subseteq I$ be a random subset where each vertex is included independently with probability $1/s$.
        Let moreover $A' = \{ a^{(b_1, b_2)} : \{b_1, b_2\} \in \binom{I}{2} \wedge N(a^{(b_1, b_2)})\cap B' = \{b_1, b_2\}\}$.
        Using that $a^{(b_1, b_2)} \in A \setminus F$ we get that
        \[
            \mathbb{E}[|A'|] = \sum_{\{b_1, b_2\} \in \binom{I}{2}} \frac{1}{s^2}\Big(1 - \frac{1}{s}\Big)^{|N(a^{(b_1, b_2)}) \cap I| -2} \geq \binom{|I|}{2} \cdot \frac{1}{5s^2} \geq 10^4h^2 |I|,
	        \]
            where we used that $|N(a^{(b_1,b_2)})\cap N|<s$.
	        In particular, there exists a choice of $B'$ such that for the corresponding $A'$ we have $|A'| \geq 10^4h^2 |B'|$. 

        Finally, let $H$ be a graph on the vertex set $B'$ with edges $E(H) = \{ N(a) \cap B' : a\in A' \}$.
        Then, we get that $d(H) \geq 10^4h^2$ and thus by Theorem~\ref{theorem:noninduced_subdivisions} we can find a subdivision of $K_h$ in $H$.
        It is easy to see that this translates back into an induced proper subdivision of $K_h$ in $G$.

        \medskip
        \noindent
        \textbf{Case 2:} For each $a \in A$ there is a non-edge $\{b_1^{(a)}, b_2^{(a)}\} \in \binom{N(a) \cap B}{2}$ with $|N(b_1^{(a)}) \cap N(b_2^{(a)}) \cap A| \leq \sqrt{K}$.

        In this case, we want to subsample $A$ and $B$ to get the desired sets $A'$ and $B'$.
        To that end, let first $\{a_1, \dots, a_{\ell}\} = A$ be a degeneracy ordering of $A$ such that writing $N_d(a_i) = N(a_i)\cap \{a_{i+1}, \dots, a_\ell\}$ we have $|N_d(a_i)| \leq d$ for each $i \in [\ell]$.
        Such an ordering exists since by assumption our graph $G$ is $d$-degenerate.
        Similarly, let $\{b_1, \dots, b_{\ell'}\}$ be a degeneracy ordering of $B$ such that writing $N_d(b_i) = N(b_i)\cap \{b_{i+1}, \dots,  b_{\ell'}\}$ we have $|N_d(b_i)| \leq d$ for each $i \in [\ell']$.

        We now let $p = 1/(10^3d)$ and let $B_p \subseteq B$
        be a random subset of $B$, where each vertex is included independently with probability $p$.
        Since we want to find an independent set, we let
        \[
            B' = \{b_i \in B_p: N_d(b_i) \cap B_p = \varnothing\},
        \]
        and note that indeed $B'$ is an independent set by definition.
        Similarly, let $A_p \subseteq A$ be a random subset of $A$ where each vertex is included independently with probability $p$.
        Define
        \[
            A' = \{a_i \in A_p: N_d(a_i) \cap A_p = \varnothing \text{ and } N(a_i) \cap B' = \{b_1^{(a)}, b_2^{(a)}\}\},
        \]
        and note that $A'$ is an independent set.

        We now want to argue that there exists a choice of $A', B'$ such that $|A'| \geq K|B|/10^{11}d^3$. 
        To that end, we estimate $\Pr[a\in A']$ for each $a \in A$. Observe that $a \in A'$ if $a \in A_p$, no vertex in $N_d(a)$ is included in $A_p$, the vertices $b_1^{(a)}, b_2^{(a)}$ are included in $B_p$ and no vertex in $(N(a) \cap B\backslash\{b_1^{(a)}, b_2^{(a)}\}) \cup N_d(b_1^{(a)})\cup N_d(b_2^{(a)})$ is included in $B_p$. Since $b_1^{(a)}b_2^{(a)}$ is not an edge, all of these events are independent, and thus we have 
        \[
            \Pr[a \in A'] \geq p \cdot(1-p)^{|N_d(a)|} \cdot p^2 \cdot (1-p)^{|(N(a) \cap B) \cup N_d(b_1^{(a)})\cup N_d(b_2^{(a)})|-2}\geq p^3(1-p)^{2\cdot 10^3d} \geq 1/(10^{11}d^3),
	        \]
	        where we used that $p = 1/(10^3d)$ to bound $(1-p)^{2\cdot 10^3d}\geq e^{-3}\geq 1/30$.
        In particular, 
        \[
            \mathbb{E}[|A'|] \geq |A| /(10^{11}d^3) \geq K|B|/(10^{11}d^3),
        \]
        and so there exists an outcome of $A', B'$ such that $|A'| \geq K|B|/(10^{11}d^3)$.

        Finally, let $H$ be the graph on the vertex set $B'$ with $E(H) = \{N(a) \cap B': a \in A'\}$.
        Since $b_1^{(a)}, b_2^{(a)}$ have at most $\sqrt{K}$ common neighbours in $A$, each edge of $H$ is defined by at most $\sqrt{K}$ vertices in $A'$. Hence, we have $|E(H)| \geq |A'|/\sqrt{K} \geq K|B|/(10^{11}d^3\sqrt{K})$. In particular $d(H) \geq 10^4h^2$.
        By Theorem~\ref{theorem:noninduced_subdivisions} we can therefore find a subdivision of $K_h$ in $H$.
        It is easy to see that this gives an induced proper subdivision of $K_h$ in $G$.
    \end{proof}
\end{lemma}

\section{Nearly-regular graphs}\label{sec:nearly_regular}
\subsection{Finding induced $K_{2,\log d}$-free subgraphs in $K_{s, t}$-free graphs}
In this section, we want to show that a nearly-regular $K_{s, t}$-free graph $G$ with average degree $d$ contains an induced $K_{2, \log d}$-free subgraph with average degree $\Omega(d^{1/(s-1)}/\log d)$.
In fact, we show that, after some cleaning, a random subset of vertices of $G$ has such a property.

More specifically, we set $p = \Theta(d^{-1 + 1/(s-1)} /\log d)$ and let $G_p \subseteq G$ be an induced subgraph where each vertex is included independently with probability $p$.
The random graph $G_p$ has the desired average degree with high probability, and so our main goal will be to eliminate the copies of $K_{2, \log d}$ by removing a small portion of vertices.

In other words, our goal is to guarantee that no pair $u, v\in G_p$ has codegree $d_G(u, v)\geq \log d$, and we split the pairs $u, v$ into two types, based on their codegree in the graph $G$.  
First, note that by K\H{o}v\'ari-S\'os-Tur\'an for each $u \in V(G)$ there are at most $O(d^{1-1/(s-1)})$ vertices $v$ having $d_G(u, v)\geq \Omega(d^{1-1/(s-1)})$.
In particular, the expected number of such high co-degree pairs that make it into $G_p$ is small --- and we will remove a vertex from each such pair.
Second, note that for each pair $u, v$ with co-degree $O(d^{1-1/(s-1)})$ it is very unlikely that $\log d$ of their common neighbours make it into $G_p$ --- and so again we are able to remove a vertex from each such $u, v$ pair to get rid of all $K_{2, \log d}$'s. 

Let us now make this sketch more formal.
\begin{lemma}\label{lemma:reduction_to_k2t}
    Let $K \geq 1$ and $t \geq s \geq 3$.
    Let $G$ be a $K_{s,t}$-free graph with average degree $d \geq e^{Ke^{15}} (10^4sKt)^{2s}$ and $\Delta(G) \leq Kd$.
    Then, there exists an induced $K_{2, \log d}$-free subgraph $H \subseteq G$ with average degree $d(H) \geq d^{1/(s-1)}/(40Kt \log d)$ and $\Delta(H) \leq 16Kd(H)$.
\end{lemma}
\begin{proof}
    Let $n = |V(G)|$ and set $p = 1/(5Ktd^{1-1/(s-1)}\log d)$.
    We will take $V_p \subseteq V(G)$ to be a random subset where each vertex is included independently with probability $p$ --- and write $G_p = G[V_p]$.
    
    We want to obtain the graph $H$ by deleting some vertices from $G_p$.
    To that end, let $T = 5Ktd^{1-1/(s-1)}$ and define 
    \[
        A_{h} = \{ \{u, v \} \in \binom{V_p}{2}: d_G(u,v) \geq T\}
    \]
    and
    \[
        A_{l} = \{ \{u,v\} \in \binom{V_p}{2} : d_G(u,v) < T \text{ and } d_{G_p}(u, v) \geq \log d\}.
    \]
    We note that by deleting one vertex from each pair in $A_h \cup A_l$ from $G_p$ we obtain a $K_{2, \log d}$-free graph.
    Finally, let also 
    \[
        B = \{ v \in V_p: d_{G_p}(v) \geq 2Kpd\}.
    \]
    By deleting $B$ from $G_p$ we will guarantee that $\Delta(H) \leq 16Kd(H)$.
    
    We will let $H$ to be the graph obtained from $G_p$ by deleting one vertex from each pair in $A_h \cup A_l$ and the vertices in $B$ --- we want to show that with positive probability it satisfies the conditions of the lemma.
    To achieve that, we will define four good events under whose intersection this holds and show that each of these events is satisfied with probability at least $4/5$ --- so that their intersection holds with positive probability. We require that: 1) $pn/2 \leq |V_p| \leq 2pn$ and $e(G_p) \geq p^2dn/4$, 2) $|B| \leq n/(Kd^{10})$, 3) $|A_l| \leq n/d^8$, and 4) $|A_h| \leq np/(100K)$.

    Since $|V_p|\sim {\rm Bin}(n, p)$, from the Chernoff bound (see e.g. Corollary 2/3 in \cite{JLR}) we have that $\Pr\big[\big||V_p|-np\big|>a np\big]\leq 2\exp(-a^2 np/3)$ for all $a\in (0, 3/2)$. Applying this with $a=1/2$, we get that $\Pr\big[\big||V_p|-pn\big|>pn/2\big]\leq 2\exp(-pn/12)<1/10$, where the last inequality follows since $pn\geq pd\geq \Omega(d^{1/(s-1)}/\log d)$.

    On the other hand, we have $\bE[e(G_p)]=p^2e(G)\geq p^2dn/2$ and ${\rm Var}[e(G_p)]\leq \sum_{e_1, e_2}(\bE[\mathbf{1}_{e_1\in G_p}\mathbf{1}_{e_2\in G_p}]-p^4)$. The only terms with a nonzero contribution are those where $e_1, e_2$ intersect, and each such term contributes a factor of $p^3-p^4\leq p^3$. Since there are at most $e(G)\cdot 2\Delta(G) \leq Kd^2n$ such pairs, we get ${\rm Var}[e(G_p)] \leq Kp^3d^2n$. Thus, 
    \[    \frac{{\rm Var}[e(G_p)]}{\bE[e(G_p)]^2} \leq \frac{Kp^3d^2n}{p^4d^2n^2/4} = \frac{4K}{pn} \leq 1/10,\]
    where the last inequality follows as before from $pn\geq \Omega(d^{1/(s-1)}/\log d)$. By Chebyshev's inequality, we get that $\Pr[e(G_p) < p^2dn/4] \leq 1/10$ and so by the union bound with the previous event we get that with probability at least $4/5$ we have both $pn/2 \leq |V_p| \leq 2pn$ and $e(G_p) \geq p^2dn/4$.
    
    To show that with probability at least $4/5$ we have $|B| \leq n/(Kd^{10})$, we use Chernoff's inequality which states that $\Pr[{\rm Bin}(m, p)\geq mp+a]\leq \exp(-a^2/2(mp+a/3))$. So, for each $v\in V(G)$ we have
    \[
        \Pr[v \in B] = \Pr\big[|N_{G}(v)\cap V_p| \geq p|N_G(v)|+Kpd\big] \leq \exp\Big({-\frac{(Kpd)^2}{2(p|N_G(v)|+Kpd/3)}}\Big)\leq 1/(5Kd^{10}),
    \]
    where we used that $|N_G(v)| \leq Kd$ for each $v \in V(G)$. Since $\bE[|B|] \leq n/(5Kd^{10})$, by Markov's inequality we get that $|B| \leq n/(Kd^{10})$ with probability at least $4/5$.

    To get a bound on $|A_l|$ we note that for each pair $u,v \in V(G)$ with $1 \leq d_G(u, v) \leq T$ the probability that $d_{G_p}(u, v) \geq \log d$ is at most
    \[
        \binom{T}{\log d} p^{\log d} \leq \left(\frac{eTp}{\log d}\right)^{\log d} \leq \left( \frac{e}{\log^2d} \right)^{\log d} \leq 1/(10K^2d^{10}).
    \]
    In particular, since for each $v \in V(G)$ there are at most $2(\Delta(G))^2 \leq 2K^2d^2$ vertices $u \in V(G)$ with $d(u, v) \geq 1$, we have $\mathbb{E}[|A_{l}|] \leq n / 5d^{8}$ and so by Markov's inequality we get that $|A_l| \leq n/d^8$ with probability at least $4/5$.

    Finally, to get a bound on $|A_h|$ we first fix a vertex $v \in V(G)$ and want to argue that there are less than $T/2$ vertices $u \in V(G) \setminus\{v\}$ with $d_G(u, v) \geq T$.
    Suppose for contradiction that this is not the case and let $U \subseteq V(G) \setminus \{v\}$ be a set of $T/2$ vertices such that $d_G(u,v) \geq T$ for each $u \in U$ and consider the induced bipartite graph $G_v = G[U, N(v) \setminus U]$.
    In $G_v$, there is no copy of $K_{s-1, t}$ with the $s-1$ vertices in $U$ --- as together with $v$ this would give us a $K_{s, t}$ in $G$.
    Therefore, using that $d_{G_v}(u) \geq T/2$ for each $u \in U$ we get
    \[  
        (t-1)\binom{|U|}{s-1} \geq \sum_{w \in N(v) \setminus U} \binom{d_{G_v}(w)}{s-1} \geq |N(v) \setminus U| \cdot \binom{T^2/(4|N(v) \setminus U| )}{s-1}.
    \]
    On the other hand, since $|N(v)| \leq Kd$, we get that
    \[
        |N(v) \setminus U| \cdot \binom{T^2/(4|N(v) \setminus U| )}{s-1} \geq Kd\binom{25Kt^2d^{1-\frac{2}{s-1}}/4}{s-1} > (t-1)\left( \frac{e T/2}{s-1} \right)^{s-1}\geq (t-1)\binom{|U|}{s-1},
    \]
    a contradiction.
    
    In particular, we get that $\mathbb{E}[|A_h|] \leq nTp^2 \leq np/\log d$.
    Therefore, by Markov's inequality, with probability at least $4/5$ we get $|A_h| \leq np/(100K)$.

    We conclude that there must exist a choice of $G_p$ for which the four above events hold.
    For such a $G_p$, we let $H$ be a graph obtained from $G_p$ by deleting a vertex from each pair in $A_l \cup A_h$ and all vertices in $B$.
    Then, by construction $H$ is $K_{2, \log d}$-free and $\Delta(H) \leq 2Kpd$.
    Moreover,
    \[
        e(H) \geq p^2dn/4 - |B|Kd - (|A_h| + |A_l|)2Kpd \geq p^2dn/8.
    \]
    and therefore, using $|V(H)| \leq 2pn$ we get that $d(H) \geq pd/8$.
    In particular, $H$ satisfies the conditions of the lemma.
\end{proof}

\subsection{Sublinear expanders}\label{section:sublinear_expanders}
In this section, we give the definition of a sublinear expander.
Sublinear expanders were first defined by Koml\'os and Szemer\'edi, who used them to show Theorem~\ref{theorem:noninduced_subdivisions}, i.e. that every graph with average degree $\Omega(h^2)$ contains a subdivision of $K_h$.
A graph is an expander if the size of the neighbourhood $N(U)$ of every set $U \subseteq V(G)$ is at least $f(|U|) \cdot |U|$ for some function $f$, and the term sublinear refers to the fact that we will take $f$ to be (slowly) decreasing in $|U|$.

\begin{definition}[Sublinear expander]
    Let $r > 0$. We say that a graph is an $r$-expander if for every $U \subseteq V(G)$ with $r/2 \leq |U| \leq n/2$ we have
    \[
        |N(U)| \geq \frac{|U|}{10 \log^2 (15|U|/r)}.
    \]
\end{definition}

A very useful feature of sublinear expanders is that every graph $G$ contains a sublinear expander $H$ with roughly the same average degree as $G$ as a subgraph.

\begin{theorem}[\cite{KomlosSzemeredi96}]\label{theorem:pass_to_expander}
    Let $r > 0$ and let $G$ be a graph.
    Then, there is an induced subgraph $H \subseteq G$ which is an $r$-expander such that $d(H) \geq d (G)/2$ and $\delta(H) \geq d(H)/2$.
\end{theorem}

Let us mention a property of sublinear expanders which makes them very useful for embedding problems: if $G$ is a sublinear expander on $n$ vertices and $U_1, U_2, F \subseteq V(G)$ are such that $|U_1|, |U_2| \gg |F|$ then there exists a path from $U_1$ to $U_2$ of length $O(\log^3 n)$ avoiding the set $F$. This allows one to build the respective structure step-by-step, while avoiding the previously used vertices at every step. In fact, we will use a slightly stronger property --- to explain why, let us give a sketch of how one could prove that every sublinear expander with average degree $\Omega(h^2 \log^3h)$ contains a subdivision of $K_h$.

We start by picking arbitrary vertices $v_1, \dots, v_h \in V(G)$.
To find a subdivision of $K_h$ in $G$ we want to find internally vertex-disjoint paths $P_{ij}$ of length $O(\log^3n)$ from $v_i$ to $v_j$ for each pair $v_i, v_j$.
We will do that greedily --- at any step we will let $F$ be all vertices of the already found paths and set $U_1 = N(v_i)$ and $U_2 = N(v_j)$ for some not-yet connected pair $v_i, v_j$.
Note that if $n \leq poly(h)$, then $\log n \leq O(\log h)$ and since $|F| \leq h^2 \log^3h \ll |U_1|, |U_2|$, we will be able to find the next path $P_{ij}$ using sublinear expansion.
If however $n \gg poly(h)$, then the set $F$ might be too large for us to do so --- and in fact we might then use up the whole neighbourhood of some of the $v_i$'s before we get the chance to connect them to all other vertices.

A standard way of dealing with this issue is to present a separate argument for this very sparse case. But in this paper we intend to instead take a unified approach. In order to avoid exhausting the neighbourhoods of $v_1, \dots, v_h$, we construct paths which do not use many vertices close to $v_1, \dots, v_h$. To make this precise, we introduce the following definition.

\begin{definition}[A drifting-away path]
    Let $G$ be a graph and $U \subseteq V(G)$.
    We say that a path $P = v_1\dots v_k$ in $G$ is \textit{drifting away} from $U$ if for all $\ell \in [k]$ we have
    \[
        {\rm dist}(v_{\ell}, U) \geq \Big\lfloor \frac{\ell^{1/3}}{20 \log \Delta(G)} \Big\rfloor.
    \]
\end{definition}

Note that if a path $P=v_1 \dots v_k$ is drifting away from $U$, then for any $t$, only the vertices $v_\ell$ with $\ell^{1/3}/(20\log \Delta(G))\leq t+1$ can be contained in $B^{(t)}(U)$. Hence, $|B^{(t)}(U)\cap V(P)|\leq (20\log \Delta(G))^3 (t+1)^3$.

\noindent We will connect the branch vertices with paths that are drifting away from $v_1, \dots, v_h$.
This will make sure that, even though the paths themselves might be long, from the perspective of small sets around the branch vertices they will behave as if they had length $O(\log^3\Delta(G)) = O(\log^3h)$.

The following lemma serves as our analogue to the statement from before that we can find short paths avoiding a small set $F$ between any two sets $U_1, U_2$.
Here, we require the path to be drifting-away and also allow $F$ to be large if we get a bound on the number of vertices in $F$ that are not far from $U_1 \cup U_2$.
\begin{lemma}\label{lemma:drifting_away_paths}
    Let $r > 0$ and let $G$ be an $r$-expander.
    Let $U, U_1, U_2 \subseteq V(G)$ be sets such that $|U_1|, |U_2| \geq r$ and $|U| \leq r/10^3$.
    Let moreover $F \subseteq V(G)$ be a set of forbidden vertices such that $|B^{(\ell)}(U_1 \cup U_2) \cap F| \leq r(\ell+1)^3/10^{15}$ for each $\ell \geq 0$.
    Then, there exist paths $P_1, P_2$ in $G$ which are disjoint from $F$ and drifting-away from $U$ such that $P_1P_2^{-1}$ is a path from $U_1$ to $U_2$ (here, $P_1P_2^{-1}$ denotes the concatenation of the path $P_1$ with the reverse of $P_2$).
\end{lemma}
\begin{proof}
    We write $d = d(G)$, $n = |V(G)|$.
    We note that for each $\ell \geq 0$ we have
    \[
        |B^{(\ell)}(U)| \leq |U| \cdot (\Delta(G))^{\ell} \leq r(\Delta(G))^{\ell}/10^3.
    \]

    Let now $A_0 = U_1 \setminus F$ and for each $i \geq 1$ we define
    \[
        A_i = A_{i-1} \cup \left( N(A_{i-1}) \setminus (F \cup B^{(\lfloor i^{1/3}/20\log \Delta(G)\rfloor)}(U)) \right).
    \]
    Note that by definition, for each $i \geq 0$ all vertices in $A_i$ are reachable from $U_1$ by a path that is drifting away from $U$ and disjoint from $F$.
    To show that the paths $P_1$ and $P_2$ exist, we will argue that $|A_i| > n/2$ for some $i$.
    By defining $A_i'$ similarly for $U_2$, we will get that $A_i \cap A_i' \neq \varnothing$ for some $i$ --- which immediately implies the existence of suitable $P_1$ and $P_2$.

    We will argue by induction that for all $i \geq 0$ we have
    \[
        |A_i| \geq \min \left\{ n/2 + 1, re^{i^{1/3}/3}/30 \right\}.
    \]
    To that end, we first notice that $|A_0| = |U_1 \setminus F| \geq r - r/10^{15} \geq r/2$, where we used our assumption on $F$.
    Now, assume that for some $i \geq 1$ we have shown the bound for $|A_{i-1}|$.
    Notice that the sequence $|A_i|$ is non-decreasing --- therefore if $|A_{i-1}| > n/2$ then the inductive step follows immediately.
    If $|A_{i-1}| \leq n/2$, then by the $r$-expansion of $G$ we get
    \begin{align*}
        |N(A_{i-1}) \setminus (F \cup B^{(\lfloor i^{1/3}/20\log \Delta(G) \rfloor)}(U))| &\geq \frac{|A_{i-1}|}{10\log^2 (15|A_{i-1}| / r)} - |F \cap B^{(i)}(U_1)| - |B^{(\lfloor i^{1/3} / 20 \log \Delta(G) \rfloor)} (U)|\\
        &\geq \frac{|A_{i-1}|}{10\log^2 (15|A_{i-1}| / r)} - 2r(i+1)^3/10^{15} - \frac{r (\Delta(G))^{i^{1/3}/20\log\Delta(G)}}{10^3}\\
        &\geq \frac{|A_{i-1}|}{13\log^2 (15|A_{i-1}| / r)},
    \end{align*}
    where we used that $A_i \subseteq B^{(i)}(U_1)$, our assumption on $F$ and that $|A_{i-1}| \geq re^{(i-1)^{1/3}/3}/30$.
    All in all, using that $|A_i|$ is non-decreasing with $i$, we get that
    \[
        15|A_i| / r \geq \prod_{j=0}^{i-1} \left(1 + \frac{1}{13\log^2(15|A_j|/r)} \right) \geq \prod_{j=0}^{i-1} \left(1 + \frac{1}{13\log^2(15|A_i|/r)} \right) \geq e^{i/(26\log^2 (15|A_i|/r))}.
    \]
    Taking logarithms gives $\log^3 (15|A_i|/r)\geq i/26$, i.e. $\log (15|A_i|/r)\geq i^{1/3}/3$, and so $|A_i| \geq re^{i^{1/3}/3}/30$.

    To conclude, we can define the sets $A_i'$ similarly for $U_2$.
    By the same argument, there must exist indices $i, i'$ minimizing $i + i'$ such that $A_i \cap A_{i'}' \neq \varnothing$.
    Now take an $x \in A_i \cap A_{i'}'$.
    By definition, there exist paths $P_1$ from $U_1$ to $x$ and $P_2$ from $U_2$ to $x$, which are both drifting-away from $U$.
    By the minimality of $i+ i'$, they don't intersect, so $P_1P_2^{-1}$ is indeed a path. 
    In particular, $P_1$ and $P_2$ satisfy the conditions of the lemma.
\end{proof}

\subsection{Finding an induced collection of disjoint paths}

As sketched above, to find an induced subdivision of $K_h$ in a $K_{2,t}$-free (or $C_{2k}$-free) expander, we would like to first fix the branch vertices $v_1, \dots, v_h$ and then iteratively connect each pair of them using Lemma~\ref{lemma:drifting_away_paths}. In this section, we give a lemma which can be used to connect a collection of $h$ sets, which one can think of as the neighbourhoods of $v_1, \dots, v_h$ using paths in a way that forms a subdivision. We first state the lemma, and then indicate how we will apply it.

\begin{lemma}\label{lemma:collection_of_paths}
    Let $h, r \geq 1$, $a \geq 16h$ and let $G$ be an $r$-expander for $r \geq 10^{23}h^2\Delta(G)\log^3\Delta(G)$.
    Let moreover $U \subseteq V(G)$ and  $S_1, \dots, S_{2h} \subseteq U$ be disjoint sets such that for each $i \in [2h]$
    \begin{enumerate}
        \item $|U| \leq r/(10^{16})$ and $|S_i| \geq a$, 
        \item any subset $S \subseteq S_i$ with $|S| \geq |S_i| /2$ has $|N(S)| \geq r$, \emph{and}
        \item for each $v \in \left(V(G) \setminus U\right) \cup S_1 \cup \dots \cup S_{2h}$ we have $|N(v) \cap (S_1\cup \dots \cup S_{2h})| \leq ha/(10^{8}h^2\log^3\Delta(G))$.
    \end{enumerate}
    Then, there exists a subset $I \subseteq [2h]$ of size $h$ and a collection of disjoint paths $\{P_{i, j} : \{i, j\} \in \binom{I}{2}\}$ such that $P_{i, j}$ is a path from $S_i$ to $S_j$, $|P_{i, j} \cap U| =2$ and the only edges spanned by $\bigcup_{i, j} V(P_{i, j})$ are $\bigcup_{i, j} E(P_{i, j})$.
\end{lemma}

In the proof of the lemma, we find the required paths one by one, at each step making sure to avoid all neighbours of previously used vertices -- which will guarantee that we indeed find an induced collection of paths. To make sure that we have enough room to avoid the forbidden vertices, we require that the sets $S_i$ are relatively large (condition 1), sparse (condition 3) and that they robustly have many neighbours (condition 2).

During the process of finding the paths, it might happen that we end up using a large portion of some $S_i$ before we get the chance to find all the paths originating from it. Therefore, we might not be able to connect up all pairs of the $S_i$'s. To circumvent that, we will instead start with $2h$ instead of $h$ candidates and argue that we can overuse at most $h$ sets $S_i$.

Let us now say a couple of words about how we apply this lemma, in the case of $K_{2, t}$-free graphs. The main idea is to choose the sets $S_i$ as subsets of the neighbourhoods of the branch vertices $v_1, \dots, v_{2h}$, and set $U=\bigcup_{i=1}^{2h}N(v_i)$. Note that $K_{2, t}$-freeness ensures that $|N(v)\cap S_i|\leq t-1$ for any $v\neq v_i$, ensuring the third condition, while the first two condition will be ensured almost by definition.

In the case of $C_{2k}$-free expanders, we will take the sets $S_i$ to be the subsets of the second neighbourhoods of $v_i$, and so we still need to make sure that we can route the endpoints of the paths through the first neighbourhood to the respective $v_i$. This requires special care, and will be done in the next subsection, alongside some additional cleaning steps. 

\begin{proof}[Proof of Lemma~\ref{lemma:collection_of_paths}.]
    We let $G$ together with sets $S_1, \dots, S_{2h} \subseteq U \subseteq V(G)$ satisfy the conditions of the lemma and write $d = d(G)$.
    We want to find the paths connecting $S_1, \dots, S_{2h}$ one by one.
    During the process, we discard some of the sets if they lose too many vertices.
    We need to argue that we can find a suitable path in each step and that we discard at most $h$ of the sets, so that at least $h$ sets remain at the end of our process.

    More formally, for each pair $\{i, j\} \in \binom{[2h]}{2}$ we want to find an induced path $P_{i,j}$ such that
    \begin{enumerate}
        \item $P_{i, j}$ is a path whose only vertices in $U$ are its endpoints $x \in S_i$ and $x' \in S_j$ and which satisfies $|B^{(\ell)}(U)\cap V(P_{i, j})|\leq 2(20\log \Delta(G))^3 (\ell+1)^3+2$ for all $\ell\geq 0$, \emph{and}
        \item for $\{i, j\} \neq \{i', j'\}$, there are no edges crossing between $P_{i', j'}$ and $P_{i, j}$.
    \end{enumerate}
    We note that these conditions guarantee that the collection of paths we find is an induced collection of disjoint paths in $G$.

    We attempt to find the paths one by one --- in each step we pick a not-yet connected pair $\{i, j\}$ and find the path $P_{i, j}$.
    We let $\mathcal{P}_q$ denote the set of $q$ paths found up until step $q$, let $V_q = \bigcup_{P \in \mathcal{P}_q} P$ and $F_q$ be the set of neighbours of some $v\in V_q$ together with $V_q$ itself. When finding the next path, we make sure that it is disjoint from $F_q$, which makes sure that the second condition is satisfied.

    We say that a set $S_i$ becomes overused at step $q$ if $|S_i \cap F_q| \geq |S_i|/2$. If a set $S_i$ becomes overused, we immediately give up on it --- we will no longer find paths connecting it to the other vertices.

    Before arguing that we can indeed find a suitable path in each step, we first show that at any point the number of overused sets is at most $h$ --- and hence at the end of the process there will be at least $h$ of them left for us to define the set $I$.

    By the first property of the paths $P_{i, j}$, for each $\ell \geq 0$ the number of vertices in $V_q$ at distance at most $\ell$ from $U$ is at most
    \[
        2q(\ell+1)^3  \cdot 8000\log^3\Delta(G)  + 4q \leq 10^5(\ell+1)^3h^2\log^3\Delta(G).
    \]
    
    Note that by the third condition in the lemma and our bound on the number of vertices in $V_q$ at distance at most $1$ to $U$, we have
    \[
        |F_q \cap \left(S_1 \cup \dots \cup S_{2h} \right)| \leq 2q + \left(2q + 8\cdot 10^{5}h^2\log^3\Delta(G)\right) \cdot \frac{ha}{10^{8}h^2\log^3\Delta(G)} \leq ha/2.
    \]
    Since $S_i \subseteq U$, $|S_i| \geq a$ for all $i \in [2h]$ and the sets $S_i$ are pairwise disjoint, we can bound the number of overused sets at any step by $|F_q \cap \left(S_1 \cup \dots \cup S_{2h} \right)|/(a/2) \leq h$.

    Suppose that for some $0\leq q < \binom{2h}{2}$ we have found a collection $\mathcal{P}_{q}$ of paths satisfying the above conditions and let $\{i, j\} \in \binom{[2h]}{2}$ be a pair such that both $S_i$ and $S_j$ are not overused at step $q$ and that we have not found the path $P_{i, j}$ yet.
    Since $S_i$ is not overused, we have $|S_i \setminus F_{q}|\geq |S_i|/2$, and so taking $U_i = N(S_i\backslash F_q)$ we get that $|U_i| \geq r$ by our second assumption in the lemma.
    We define $U_j$ similarly for $S_j$ and again get that $|U_j| \geq r$.

    We now apply Lemma~\ref{lemma:drifting_away_paths} with the sets $U_i$ and $U_j$ and $F = F_{q} \cup U$.
    We note that, using the above observations, for each $\ell \geq 0$ the number of vertices in $F$ at distance at most $\ell$ from $U_i \cup U_j$ is at most
    \[
        |U| + 10^5(\ell+2)^3h^2\Delta(G)\log^3\Delta(G) \leq \frac{r}{10^{16}} + \frac{r + r\ell^3}{10^{16}} \leq \frac{r + r\ell^3}{10^{15}},
    \]
    where we used our assumption on $r$ and the size of $U$.
    In particular, the premise of Lemma~\ref{lemma:drifting_away_paths} is satisfied.
    We thus get paths $P_1, P_2$ that are disjoint from $F$ and drifting-away from $U$ such that $P_1P_2^{-1}$ is a path from $U_1$ to $U_2$.
    We pick $x_1 \in N(P_1) \cap S_i\backslash F_q$ and $x_2 \in N(P_2) \cap S_j\backslash F_q$.
    Since the induced subgraph on $V(P_1)\cup V(P_2)\cup \{x_1, x_2\}$ is connected, we may take $P_{i, j}$ to be the shortest path between $x_1$ and $x_2$ in this subgraph, and note it must be induced due to minimality. Finally, this path satisfies the first property since $P_1$ and $P_2$ are drifting away from $U$ and so \[|B^{(\ell)}(U)\cap V(P_{i, j})|\leq |B^{(\ell)}(U)\cap V(P_1)|+|B^{(\ell)}(U)\cap V(P_2)|+2\leq 2(20\log \Delta(G))^3 (\ell+1)^3+2.\]
    
    At the end of the process, we have found a suitable path between all pairs of non-overused sets.
    Since there are at least $h$ non-overused sets, we can pick $I \subseteq [2h]$ of size $h$ such that the corresponding subset of the paths satisfies the conditions of the lemma.
\end{proof}

\subsection{Finding induced subdivisions}
In this section, we use Lemma~\ref{lemma:collection_of_paths} to prove the following two lemmas, which say that we can find an induced subdivision of $K_h$ in nearly-regular expanders.
\begin{lemma}\label{lemma:k2t_induced_subdivisions}
    For each $K$ there exists a constant $C_{\ref{lemma:k2t_induced_subdivisions}}$ such that the following holds for all $h, t \in \mathbb{N}$.
    Let $G$ be a $K_{2, t}$-free graph with $d(G) \geq C_{\ref{lemma:k2t_induced_subdivisions}}th^2\log^3(ht)$ and $\Delta(G) \leq Kd(G)$.
    Then, $G$ contains an induced proper subdivision of $K_h$.
\end{lemma}
\begin{lemma}\label{lemma:c2k_induced_subdivisions}
    For each $K, k \geq 3$ there exists a constant $C_{\ref{lemma:c2k_induced_subdivisions}}$ such that the following holds for all $h \in \mathbb{N}$.
    Let $G$ be a $C_{2k}$-free graph with $d(G) \geq C_{\ref{lemma:c2k_induced_subdivisions}}h\log^3h$ and $\Delta(G) \leq Kd(G)$.
    Then, $G$ contains an induced proper subdivision of $K_h$.
\end{lemma}

The proof of the $K_{2,t}$-free case is rather straightforward --- we pick the branch vertices $v_1, \dots, v_{2h}$ while making sure that, after some small cleaning, their first neighbourhoods $S_1, \dots, S_{2h}$ are disjoint.
By $K_{2,t}$-freeness, the degrees into each $S_i$ are at most $t-1$.
Hence, we will get the collection of paths from Lemma~\ref{lemma:collection_of_paths}, which will then directly yield an induced subdivision of $K_h$.

\begin{proof}[Proof of Lemma~\ref{lemma:k2t_induced_subdivisions}]
    Fix $K \geq 1$.
    We let $C_{\ref{lemma:k2t_induced_subdivisions}} \geq 10^{18}K$ be large enough compared to $K$ such that for all $h, t \in \mathbb{N}$ and all $d_0 \geq C_{\ref{lemma:k2t_induced_subdivisions}} th^2\log^3ht$ we have
    \[
        \frac{d_0^2}{128t} \geq 10^{23}h^2(Kd_0)\log^3(Kd_0).
    \]
    In the following, we fix $h, t \in \mathbb{N}$ and let $G$ be a graph satisfying the conditions of the lemma.

    First, we let $r = (d(G)/2)^2/(32t) \geq 10^{23}h^2\Delta(G)\log^3\Delta(G)$.
    By Theorem~\ref{theorem:pass_to_expander} there exists an induced $G' \subseteq G$ which is an $r$-expander such that $d(G') \geq d(G)/2 \geq (C_{\ref{lemma:k2t_induced_subdivisions}}/2)th^2\log^3(ht)$ and $\delta(G') \geq d(G')/2$.
    We write $n = |V(G')|$, $d = d(G')$ and note that $n \geq d^2/(4t)$ --- otherwise $G'$ would contain a $K_{2, t}$ by the K\H{o}vari-S\'os-Tur\'an theorem. Let us also note that $\Delta(G')\le 2Kd$. 

    We begin by picking $v_1, \dots, v_{2h} \in V(G')$ arbitrarily at pairwise distance at least 2, which will serve as our branch vertices (this is possible, since $2h(\Delta(G')+1) \leq 5Khd < d^2/(4t)$).
    Note moreover that for every $i$ and every $v \in V(G')\backslash \{v_i\}$ we have $|N_{G'}(v) \cap N_{G'}(v_i)| <t$ by $K_{2, t}$-freeness.
    Therefore, we can pick subsets $S_i \subseteq N_{G'}(v_i)\backslash \bigcup_{j\neq i}N_{G'}(v_j)$ for all $i \in [2h]$ such that $|S_i| \geq \delta(G') - 2ht \geq d/4 \geq 16h$ and $S_1, \dots, S_{2h}$ are pairwise disjoint.
    These will be the starting sets for the induced paths connecting the branch vertices.

    Now, let $U = \bigcup_{i \in [2h]}B^{(1)}_{G'}(v_i)$ and note that we have $|U| \leq 2h(1 +\Delta(G')) \leq r/(10^{16})$.
    Moreover, for each $i \in [2h]$ and set $S \subseteq S_i$ with $|S| \geq |S_i|/2 \geq d/8$ we have
    \[
        |N_{G'}(S) \setminus S_i| \geq |S|\cdot (\delta(G') - t)/t \geq d^2/(32t) \geq r,
    \]
    where we used that for all $v \in V(G')\backslash \{v_i\} $ we have $|N_{G'}(v) \cap S_i|\le |N_{G'}(v)\cap N_{G'}(v_i)| <t<\delta(G')/2$.
    Since additionally $|N_{G'}(w) \cap (S_1 \cup \dots \cup S_{2h})| \leq 2h \cdot t \leq h(d/4)/(10^8h^2\log^3\Delta(G))$ for all $w \neq v_1, \dots, v_{2h}$, by Lemma~\ref{lemma:collection_of_paths} there exists an $I \subseteq [2h]$ of size $h$ and an induced collection $\{P_{i, j} : \{i, j\} \in \binom{I}{2}\}$ of disjoint induced paths in $G'$ such that $P_{i, j}$ is a path from $S_i$ to $S_j$ such that $|P_{i, j} \cap U| = 2$.
    
    Since all the neighbours of $v_i$ are contained in $U$, and only vertices from $U$ used on the paths $P_{i, j}$ are their endpoints, this collection together with the vertices $\{v_i : i \in I\}$ induces a proper subdivision of $K_h$ in $G$ (as $G'$ is an induced subgraph of $G$). 
\end{proof}

In the $C_{2k}$-free case, we would like to take the second neighbourhoods as the sets $S_1, \dots, S_{2h}$ for Lemma~\ref{lemma:collection_of_paths}.
However, we need to be careful that we can later extend the paths given by the lemma to the respective $v_i$.
To achieve that, we will first clean the neighbourhoods around the $v_i$'s to ensure that 1) there are no vertices in $S_i$ having large codegree with $v_i$ and that 2) the sets $S_i$ are disjoint.
Then, we will invoke Lemma~\ref{lemma:collection_of_paths} on an auxiliary graph in which we add an edge between $x, y \in S_i$ if connecting both $x$ and $y$ to $v_i$ might not be possible in an induced way.
Since the paths we get form an induced collection of disjoint paths, it will in particular make sure that we will get no conflicts while connecting the endpoints to the respective $v_i$.

To do the cleaning, we will need to use that small balls around each vertex are sparse, in that they do not contain a long path. This is given by the following lemma.
\begin{lemma}\label{lemma:no_long_path}
    Let $k \geq 3$, let $G$ be a $C_{2k}$-free graph and let $v \in V(G)$.
    Then, there is no $P_{2k}$ in either $G[N(v)]$ or $G[N(v),N^{(2)}(v)]$. Additionally, there is no $P_{4k-8}$ in $G[N^{(2)}(v)]$ and no $P_{4k-6}$ in $G[N^{(2)}(v), N^{(3)}(v)]$.
\end{lemma}
\begin{proof}
    We first note that if $G[N(v), N^{(2)}(v)]$ contained a $P_{2k}$ then it would also contain a $P_{2k-2}$ with both ends in $N(v)$ --- which would immediately give us a $C_{2k}$ in $G$. Similarly, there can not be any $P_{2k}$ in $G[N(v)]$ (in fact, we cannot even have a $P_{2k-2}$).

    Suppose therefore there is a path $v_1v_2\dots v_{4k-7}$ with $v_i \in N^{(2)}(v)$ for each odd $i \in [4k-7]$ and $v_i\not\in N(v)$ for even $i\in [4k-7]$, which would by implied by either a $P_{4k-8}$ in $G[N^{(2)}(v)]$ and a $P_{4k-6}$ in $G[N^{(2)}(v), N^{(3)}(v)]$.
    For each odd $i \in [4k-7]$ pick an $x_i \in N(v_i) \cap N(v)$ arbitrarily.
    Then, we must have $x_1 = x_{2k-3} = x_{4k-7}$, as otherwise we could find a $C_{2k}$ using the vertex $v$.
    Similarly, we must have $x_{3} = x_{2k-1} \neq x_1$, as otherwise $v_1v_2\dots v_{2k-1}x_1v_1$ would form a $C_{2k}$.
    But then, $v_1v_2v_3x_3v_{2k-1}v_{2k}\dots v_{4k-7}x_1v_1$ is a $C_{2k}$ in $G$ --- a contradiction.
\end{proof}

With that in hand, we are ready to find induced subdivisions in $C_{2k}$-free nearly-regular expanders.
\paragraph{\textit{Proof of Lemma~\ref{lemma:c2k_induced_subdivisions}.}}
    Fix $K, k \geq 3$.
    We will let $C_{\ref{lemma:c2k_induced_subdivisions}}$ be large enough compared to $k, K$ such that for all $h$ and all $d_0 \geq C_{\ref{lemma:c2k_induced_subdivisions}}h\log^3h$ we have
    \[  
        \frac{(d_0/2)^3}{10^4k^2} \geq \max\{ 10^{23}h^2(10^{12}K^5k^4d_0)\log^3(10^{12}K^5k^4d_0), 10^{17}h(Kd_0)^2\}.
    \]
    In the following, we fix $h \in \mathbb{N}$ and let $G$ be a graph satisfying the conditions of the lemma.

    First, we let $r = (d(G)/2)^3/(10^4k^2)$.
    By Theorem~\ref{theorem:pass_to_expander} there exists an induced $G' \subseteq G$ that is an $r$-expander such that $d(G') \geq d(G)/2$ and $\delta(G') \geq d(G')/2$.
    We restrict our attention to this $G'$ and write $n = |V(G')|$, $d = d(G')$ and note that since $\operatorname{ex}(n, C_{2k}) \leq 100kn^{1 + 1/k}$, we must have $n \geq (d/200k)^k$.

    The following claim will give us suitable candidates for the branch vertices and the subsets of their second neighbourhoods we will use to apply Lemma~\ref{lemma:collection_of_paths}.
    \begin{claim}\label{claim:c2k_branch_vertices}
        There exist $v_1, \dots, v_{2h} \in V(G')$ and subsets $I_i \subseteq N_{G'}(v_i)$ and $S_i \subseteq N_{G'}(I_i) \cap N^{(2)}_{G'}(v_i)$ such that for each pair of distinct $i, j\in [2h]$
    \begin{enumerate}
        \item for each $u \in I_i$ we have $|N_{G'}(u) \cap N_{G'}(v_i)| \leq 10^4K^2k^2$,
        \item $|S_i| \geq d^2/(400k)$ and for each $w \in S_i$ we have $|N_{G'}(w) \setminus N_{G'}^{(3)}(v_i)| \leq 10^6K^2 k^2$, \emph{and}
        \item $B^{(1)}(\{v_i\} \cup I_i) \cap (\{v_j\} \cup I_j \cup S_j) = \varnothing$.
    \end{enumerate}
    \end{claim}

    Before proving the claim, we want to conclude the proof of the lemma assuming that the claim holds. We will apply the Lemma~\ref{lemma:collection_of_paths} to the collection of sets $S_i$ with $U=\bigcup_{i=1}^{2h} B_{G'}^{(2)}(v_i)$. However, connecting the paths of this collection to the branch vertices $v_i$ requires care -- a priori, it may not be possible to connect any two vertices $v, v'\in S_i$ to $v_i$ using disjoint paths with no crossing edges.

    To get around this issue, we define an auxiliary graph $H$ by adding suitable conflict edges to $G'$, and apply Lemma~\ref{lemma:collection_of_paths} to $H$.
    We say that vertices $v, v'\in S_i$ are in conflict if there is a vertex $y\in I_i$ adjacent to both, or if there is an edge $yy'$ in $I_i$ such that $yv, y'v'\in E(G')$. More formally, the set of conflicting pairs is
    \[
        E = \!\!\bigcup_{i\in [2h]} \!\Big\{\{v, v'\} \!\in\! \binom{S_i}{2}\! : \text{there exist } y, y'\!\in I_i\text{ such that } yv, y'v'\!\in E(G') \text{ and } \big(y=y' \text{ or }yy'\in E(G')\big)\Big\}.
    \]
    Define the graph $H$ with $V(H) = V(G')$ and $E(H) = E(G') \cup E$. A vertex $v$ is in conflict with at most
    \[
        10^6K^2k^2 \cdot (1 + 10^4K^2k^2)\cdot \Delta(G') \leq 10^{11}K^4k^4\Delta(G)
    \]
    other vertices $v'$. To see this, note that given $v$, there are at most $10^6K^2k^2$ choices for $y$ (due to property 2 from Claim~\ref{claim:c2k_branch_vertices}), and at most $10^4K^2k^2+1$ choices for $y'$ (due to property 1), and at most $\Delta(G')$ choices for $v'$. So, we have that $\Delta(H) \leq \Delta(G') + 10^{11}K^4k^4 \Delta(G) \leq 10^{12}K^5k^4d$.

    Let us justify that Lemma~\ref{lemma:collection_of_paths} applies to the graph $H$ with the sets $S_1, \dots, S_{2h}$ and $U = \bigcup_{i \in [2h]} B^{(2)}_{G'}(v_i)$. Note that
    \[
        |U| \leq 2h (1 + \Delta(G') + \Delta(G')^2) \leq r/(10^{16}).
    \]
    Since the graph $G[N^{(2)}(v_i), N^{(3)}(v_i)]$ is bipartite and is $P_{4k}$-free its average degree is at most $4k$ and so $|S| \geq |S_i|/2$, we have
    \[
        |N_{G'}(S)|\geq |N_{G'}(S)\cap N^{(3)}_{G'}(v_i)| \geq |S| \cdot (\delta(G') - 10^6K^2k^2) / (4k) \geq d^3/(10^4k^2) \geq r.
    \]
    Finally, for each $v \in V(H)$ we have $|N_H(v) \cap (S_1 \cup \dots \cup S_{2h})| \leq \Delta(H) \leq (hd^2/400k)/(10^{8}h^2\log^3\Delta(H))$.
    Hence, the premise of Lemma~\ref{lemma:collection_of_paths} holds, and we can find an $I \subseteq [2h]$ of size $h$ and an induced collection of disjoint induced paths $\{P_{i, j} : \{i, j\} \in \binom{I}{2}\}$ such that $P_{i, j}$ is a path from $S_i$ to $S_j$.
    We note that since each edge in $E$ lies fully inside $U$, each $P_{i, j}$ is also an (induced) path in $G'$.

    To conclude, note that any path $P_{i, j}$ can be extended to be a path from $v_i$ to $v_j$ by adding one vertex from $I_i \cap N_{G'}(P_{i, j})$ and $I_j \cap N_{G'}(P_{i, j})$ to each of its ends. 
    By the definition of $E$, no matter how the paths $P_{i, j}$ are extended, the resulting paths together with $\{v_i : i \in I\}$ form an induced proper subdivision of $K_h$ in $G$.

    \begin{proof}[Proof of Claim~\ref{claim:c2k_branch_vertices}]
        We will first find vertices $v_1, \dots, v_{2h} \in V(G')$ and corresponding sets $I_i$ and $S_i'$, such that
        \begin{enumerate}
            \item for each $u \in I_i$ we have $|N_{G'}(u) \cap N_{G'}(v_i)| \leq 10^4K^2k^2$,
            \item $|S_i'| \geq d^2/(200k)$ and for each $w \in S_i'$ we have $|N_{G'}(w) \setminus N_{G'}^{(3)}(v_i)| \leq 10^6 K^2k^2$, \emph{and}
            \item $B^{(1)}_{G'}(\{v_i\} \cup I_i) \cap (\{v_j\} \cup I_j) = \varnothing$ and $|N_{G'}(w) \cap S_i'| \leq 4k$ for all $i\neq j$ and $w \in \{v_j\} \cup I_j$.
        \end{enumerate}
        Given this we can clean these small intersections by taking $S_i:= S_i'\setminus \bigcup_{j\neq i}N(\{v_j\}\cup I_j)$.
        Then, we get $|S_i|\ge d^2/(200k)- 2h\cdot 4k(1+\Delta(G))\ge d^2/400k$ (since $10^4 k^2Kh\le d_0/2\leq d$), completing the proof. Thus, it remains to find these initial sets $I_i$ and $S_i'$. 

        We find these sets iteratively --- fix some $i \in [2h]$ and suppose that we have found $v_j, I_j, S_j'$ for all $j < i$.
        To find $v_i, I_i, S_i'$, we define $D_j = \{x \in V(G') \setminus B_{G'}^{(2)}(v_j): |N_{G'}(x) \cap N^{(2)}_{G'}(v_j)| \geq 4k\}$ for $j < i$.
        Since $G[N^{(2)}_{G'}(v_j), D_j]$ contains no path of length $4k-1$ due to Lemma~\ref{lemma:no_long_path} we have $|D_j| \leq |N_{G'}^{(2)}(v_j)| \leq \Delta(G)^2\leq 4K^2d^2$. Indeed, otherwise we get a bipartite graph whose larger part has all the vertices having degree at least $4k$, which has to contain $P_{4k-1}$, contradiction.
        

        Let $F = \bigcup_{j < i} B_{G'}^{(2)}(v_j) \cup D_j$ and note that $|F|\le i (1+Kd +4K^2d^2+4K^2d^2)\le 2h\cdot 9K^2d^2\le n/8K$.
        If we remove $F$ from $G$, the remaining graph has average degree at least $2{(nd/2-\Delta(G)\cdot n/8K)}/{n}\geq d/2$. 
        So there exists $W \subseteq V(G')\backslash F$ such that for $G[W]$ we have $\delta(G[W]) \geq d/4$ and $\Delta(G[W]) \leq 2Kd$.

        Pick an arbitrary vertex $v_i \in W$.
        Let 
        \begin{align*}
            R_1 &= \{u \in N_{G'}(v_i): |N_{G'}(v_i) \cap N_{G'}(u)| \geq 10^4K^2k^2\},\\
            R_2 &= \{w \in N_{G'}^{(2)}(v_i): |N_{G'}(w) \cap N_{G'}(v_i)| \geq 10^4K^2k^2\},\\
            R_3 &= \{ w\in N_{G'}^{(2)}(v_i): |N_{G'}(w) \cap N_{G'}^{(2)}(v_i)| \geq 10^5K^2k^2\},
        \end{align*}
         and define $I_i= N_{G'}(v_i)\cap W\backslash R_1$ and $S_i' = (N_{G'}(I_i) \cap N_{G'}^{(2)}(v_i)\cap W) \setminus (R_2 \cup R_3)$. 
         
         Since $I_i, S_i'$ are disjoint from $F$, we have $B^{(1)}_{G'}(\{v_i\} \cup I_i) \cap (\{v_j\} \cup I_j) = \varnothing$ for all $j<i$. Moreover, we have that $N_{G'}(w)\cap S_i' = \emptyset$ for $w\in \bigcup_{j<i} \{v_j\}\cup I_j$ (as $N_{G'}(w)\subset F$). Meanwhile, for $w\in \{v_i\}\cup I_i$, $w\in W$ implies that $|N_{G'}(w)\cap S_j'|< 4k$ for all $j<i$ (as $D_j \subset F$ as well). This justifies the third condition.
         
         This definition further ensures that for $u\in I_i$ we have $|N_{G'}(u)\cap N_{G'}(v_i)|\leq 10^4K^2k^2$, justifying the first condition.
         
         Further, for $w\in S_i'$ we have $|N_{G'}(w) \setminus N_{G'}^{(3)}(v_i)|\leq |N_{G'}(w) \cap N_{G'}(v_i)| + |N_{G'}(w) \cap N_{G'}^{(2)}(v_i)| \leq 10^6 K^2k^2$. So, it only remains to show that $|S_i'|\geq d^2/(200k)$.
        
        Since $G[N_{G'}(v_i)]$ is $P_{2k}$-free, we have $e(N_{G'}(v_i)) \leq 2k|N_{G'}(v_i)|$. So, by Markov's inequality $|R_1| \leq |N_{G'}(v_i)| / (5000K^2k) \leq d / (2000Kk)$.
        
        Similarly, since $G[N_{G'}(v_i), N^{(2)}_{G'}(v_i)]$ is $P_{2k}$-free and $\delta(G') \geq d/2$, we have $|N^{(2)}_{G'}(v_i)| \geq |N_{G'}(v_i)|$ and so $e(N_{G'}(v_i), N^{(2)}_{G'}(v_i))\leq k|N_{G'}(v_i)|+ k|N^{(2)}_{G'}(v_i)|\leq 2k|N^{(2)}_{G'}(v_i)|$. Hence $|R_2| \leq |N^{(2)}_{G'}(v_i)|/(2000K^2k) \leq d^2/(500k)$. Similarly, since $G[N_{G'}^{(2)}(v_i)]$ is $P_{4k}$-free, we get $|R_3| \leq |N^{(2)}(v_i)| / (10^4K^2k) \leq d^2/(10^3k)$.
        
        Finally, by repeating the same argument for $G[W]$ and using $\delta(G[W])\geq d/4$, we get $2k|N^{(2)}_{G[W]}(v_i)|\geq e(N_{G[W]}(v_i), N^{(2)}_{G[W]}(v_i))\geq |N_{G[W]}(v_i)|\delta(G[W]) - 2e(N_{G[W]}(v_i))$. Therefore,
        \[
            |N_{G[W]}^{(2)}(v_i)| \geq (|N_{G[W]}(v_i)|\delta(G[W]) - 4k|N_{G[W]}(v_i)|)/4k \geq d^2/(80k),
        \]
        and so $|S_i'| \geq |N_{G[W]}^{(2)}(v_i)| - \Delta(G)|R_1| - |R_2| - |R_3| \geq d^2/(200k)$.\hspace{5.5cm} \qedsymbol
    \end{proof}

\vspace{-0.5cm}
\section{Proofs of two main results}\label{sec:main_proofs}
In this section, we prove Theorem~\ref{theorem:kst} and Theorem~\ref{theorem:C2k}.
With everything set up, this boils down to invoking the above lemmas in the right order and with the right parameters.
\begin{proof}[Proof of Theorem~\ref{theorem:kst}]
    Fix $s \geq 2$, and let us begin by defining the relevant parameters.
    We let $C_{\ref{lemma:dichotomy}}$ be the constant from Lemma~\ref{lemma:dichotomy} and $C_{\ref{lemma:k2t_induced_subdivisions}}$ be the constant from Lemma~\ref{lemma:k2t_induced_subdivisions} for the parameter $K_{\ref{lemma:k2t_induced_subdivisions}}:= 16C_{\ref{lemma:dichotomy}}$.
    Let $C \geq(10s)^{6s}$ be chosen such that for all $h \geq t \geq s$ and all $d \geq C t^{s-1}h^{2(s-1)}\log^{7(s-1)}d$, setting $K = 4t^2s^2d^{200s}$, we have $d \geq 2000\log^2K \left( 80C_{\ref{lemma:dichotomy}}C_{\ref{lemma:k2t_induced_subdivisions}}h^2t \log^2 d \cdot \log^3(h\log d) \right)^{s-1}$.

    In the following, we let $G$ be a $K_{s, t}$-free graph on $n$ vertices with average degree $d = d(G) \geq Ct^{s-1}h^{2(s-1)}\log^{7(s-1)}d$.
    We note that to show that $G$ contains an induced subdivision of any graph $H$ on $h$ vertices, it is sufficient to find an induced proper subdivision of $K_h$ in $G$.
    We also note that we can assume that $G$ is $d$-degenerate and that $\delta(G) \geq d/2$ --- otherwise we can pass to an induced subgraph of $G$ with larger average degree and show the statement there.

    We first apply Lemma~\ref{lemma:dichotomy} with the parameter $K = 4t^2s^2d^{200s}$.
    We get that either we can partition the vertex set of $G$ into $V(G) = A \cup B$ such that $|A| \geq K|B|$ and $e_G(A, B) \geq nd(G)/8$ or there exist an induced subgraph $G_1 \subseteq G$ with $d(G_1) \geq d(G) / (2000\log^2 K)$ and $\Delta(G_1) \leq C_{\ref{lemma:dichotomy}}d(G_1)$.
    We treat the two cases separately.

    In the former case, we first let $A' = \{ a \in A:d(a, B) \geq d/10^3 \text{ and } d(a) \leq 10^3\cdot d \}$.
    Since $G$ is $d(G)$-degenerate, the number of vertices with degree more than $10^3d$ is at most $n/10^3$.
    Similarly, since $e(G[A])\leq nd/100$, the number of vertices $a \in A$ with $d(a, A) \geq d/4$ is at most $n/4$.
    In particular, we get 
    \[  
    |A'| \geq (2/3-10^{-3}-1/4)n \geq n/4 \geq t^2s^2d^{200s}\cdot |B|,
    \]
    where we used the choice of $K$.
    Noting that $C \geq (10s)^{6s}$, by applying Lemma~\ref{lemma:lopsided_case} to the graph $G[A' \cup B]$ we can therefore find an induced proper subdivision of $K_h$ in $G$.

    In the latter case, if $s \geq 3$, we apply Lemma~\ref{lemma:reduction_to_k2t} to the graph $G_1$ to get an induced $K_{2, \log d}$-free subgraph $G_2 \subseteq G_1$ with 
    \[
        d(G_2) \geq d(G_1)^{1/(s-1)}/(40C_{\ref{lemma:dichotomy}}t \log d)\geq C_{\ref{lemma:k2t_induced_subdivisions}} h^2\cdot \log d \cdot \log^3 (h \log d)
    \]
    and $\Delta(G_2) \leq 16C_{\ref{lemma:dichotomy}} d(G_2).$
    Finally, by Lemma~\ref{lemma:k2t_induced_subdivisions} we can find an induced proper subdivision of $K_h$ in $G_2$.

    If $s = 2$, then by the same computation we can simply find an induced proper subdivision of $K_h$ in $G_1$.
\end{proof}

The proof of Theorem~\ref{theorem:C2k} follows the same lines.
\begin{proof}[Proof of Theorem~\ref{theorem:C2k}]
    Fix $k \geq 3$ --- we first set-up the relevant parameters to fix the implicit constant for the theorem.
    We let $C_{\ref{lemma:dichotomy}}$ be the constant from Lemma~\ref{lemma:dichotomy} and $C_{\ref{lemma:c2k_induced_subdivisions}}$ be the constant from Lemma~\ref{lemma:c2k_induced_subdivisions} with parameters $k_{\ref{lemma:c2k_induced_subdivisions}} := k$ and $K_{\ref{lemma:c2k_induced_subdivisions}} := 2C_{\ref{lemma:dichotomy}}$.
    We let $C = C(k) \geq 10^{10}k^2$ be chosen such that for all $d \geq Ch\log^5h$, setting $K = 8d^{200}$, we have $d \geq 2000C_{\ref{lemma:c2k_induced_subdivisions}}  h\log^2K \log^3h$. 

    In the following, we let $G$ be a $C_{2k}$-free graph on $n$ vertices with average degree $d = d(G) \geq Ch\log^5h$. 
    We note that to show that $G$ contains an induced subdivision of any graph $H$ on $h$ vertices, it is sufficient to find an induced proper subdivision of $K_h$ in $G$.
    We also note that we can assume that $G$ is $d$-degenerate and that $\delta(G) \geq d/2$ --- otherwise we can pass to an induced subgraph of $G$ with larger average degree and show the statement there.

    We first apply Lemma~\ref{lemma:dichotomy} with the parameter $K = 8d^{200}$.
    We get that either we can partition the vertex set of $G$ into $V(G) = A \cup B$ such that $|A| \geq K|B|$ and $e_G(A, B) \geq nd/8$ or that there exists an induced subgraph $G_1 \subseteq G$ with $d(G_1) \geq d(G)/(2000\log^2K)$ and $\Delta(G_1) \leq C_{\ref{lemma:dichotomy}} d(G_1)$.
    We treat the two cases separately.

    In the former case, we first let $A' = \{ a \in A:d(a, B) \geq d/10^3 \text{ and } d(a) \leq 10^3d\}$.
    Since $G$ is $d$-degenerate, the number of vertices with degree more than $10^3d$ is at most $n/10^3$.
    Similarly, since $e(G[A])\leq nd/100$, the number of vertices $a \in A$ with $d(a, A) \geq d/4$ is at most $n/4$.
    In particular, we get 
    \[
        |A'| \geq (2/3-1/10^3-1/4)n\geq n/4 \geq d^{200}\cdot |B|,
    \]
    where we used the choice of $K$.
    Using our choice of $C$, by applying Lemma~\ref{lemma:lopsided_case} to the graph $G[A' \cup B]$ we can therefore find an induced proper subdivision of $K_h$ in $G$.

    Otherwise, since $d(G_1) \geq C_{\ref{lemma:c2k_induced_subdivisions}}h\log^3h$, we find an induced proper subdivision in $G_1$ by Lemma~\ref{lemma:c2k_induced_subdivisions}.
\end{proof}

\section{Concluding remarks}\label{sec:concluding_remarks}

In Theorem~\ref{theorem:C2k}, we have shown that every $C_{2k}$-free graph of average degree at least $C_k h (\log h)^5$ contains an induced subdivision of $K_h$. As previously discussed, this bound is optimal up to a factor of $O((\log h)^{5})$. We think it would be interesting to remove these logarithmic factors and show that every $C_{2k}$-free graph of average degree at least $Ch$ contains an induced subdivision of $K_h$.

Theorem~\ref{theorem:kst} shows that $K_{s, t}$-free graphs of average degree $C_{s,t} h^{2(s-1)}(\log h)^{7(s-1)}$ contain induced subdivisions of $K_h$. While removing the logarithmic factors here may also be interesting, the lower bounds here are further off: the gap between upper and lower bounds is a small polynomial factor in $h$, whose exponent tends to zero when $t$ tends to infinity.

Recall that, to show the lower bound, we have constructed a $K_{s, t}$-free graph without independent sets of size say $h^2/4$, and we have argued that any induced subdivision of $K_h$ would have to contain an independent set of at least this size. This approach basically reduces the problem of proving the lower bounds to the problem of obtaining the lower bounds on the off-diagonal Ramsey numbers $R(K_{s, t}, K_{h^2/4})$. Even in the simplest case $s=t=2$, this is a well-known open problem of estimating $R(C_4, K_n)$. Thus, unless some additional properties of subdivisions are used, it seems rather hard to improve the lower bounds corresponding to Theorem~\ref{theorem:kst}.

\end{document}